\documentclass[11pt,reqno,tbtags]{amsart}
\usepackage{amssymb}
\usepackage{natbib}
\bibpunct[, ]{[}{]}{;}{n}{,}{,}
\newcommand\citetq[2]{\citeauthor{#2} \cite[{\frenchspacing #1}]{#2}} 

\numberwithin{equation}{section}
\allowdisplaybreaks

\newtheorem{theorem}{Theorem}[section]
\newtheorem{lemma}[theorem]{Lemma}
\newtheorem{proposition}[theorem]{Proposition}
\newtheorem{corollary}[theorem]{Corollary}

\theoremstyle{definition}
\newtheorem{example}[theorem]{Example}
\newtheorem{remark}[theorem]{Remark}

\theoremstyle{remark}

\newenvironment{exx}[1]
{\begin{example}[\textbf{#1}]}
{\end{example}}

\newenvironment{romenumerate}{\begin{enumerate}% gives (i), (ii) etc.
 }{\end{enumerate}}

\newcounter{oldenumi}
{\setcounter{oldenumi}{\value{enumi}}
\begin{romenumerate} \setcounter{enumi}{\value{oldenumi}}}
{\end{romenumerate}}

\newcounter{thmenumerate}
\newenvironment{thmenumerate}
{\setcounter{thmenumerate}{0}%
 \def\item{\par% \ifnum\thethmenumerate=0\else\par\fi %\noindent\fi
 \refstepcounter{thmenumerate}\textup{(\roman{thmenumerate})\enspace}}
}
{}

\newcounter{xenumerate}   %no left indentation; thus wider lines
\newenvironment{xenumerate}
  {\begin{list}
    {\upshape(\roman{xenumerate})}
    {\setlength{\leftmargin}{0pt}
     \setlength{\rightmargin}{0pt}
     \setlength{\labelwidth}{0pt}
     \setlength{\itemindent}{\labelsep}
     \setlength{\topsep}{0pt}
     \usecounter{xenumerate}} }
  {\end{list}}

\newcommand\pfitem[1]{\par(#1):}
\newcommand\pfitemx[1]{\par\textit{#1}:}

\newcommand{\refT}[1]{Theorem~\ref{#1}}
\newcommand{\refC}[1]{Corollary~\ref{#1}}
\newcommand{\refL}[1]{Lemma~\ref{#1}}
\newcommand{\refR}[1]{Remark~\ref{#1}}
\newcommand{\refS}[1]{Section~\ref{#1}}

\newcommand{\refP}[1]{Proposition~\ref{#1}}
\newcommand{\refE}[1]{Example~\ref{#1}}

\newcommand{\refApp}[1]{Appendix~\ref{#1}}
\newcommand{\refand}[2]{\ref{#1} and~\ref{#2}}

\begingroup
  \divide\count255 by 60
  \multiply\count255 by -60
  \advance\count255 by \time
  \ifnum \count255 < 10 \xdef\klockan{\the\count1.0\the\count255}
  \else\xdef\klockan{\the\count1.\the\count255}\fi
\endgroup

\newcommand{\sumk}{\sum_{k=0}^\infty}
\newcommand{\summ}{\sum_{m=0}^\infty}
\newcommand{\sumn}{\sum_{n=0}^\infty}

\newcommand{\sumni}{\sum_{n=1}^\infty}

\newcommand{\sumjj}{\sum_{j=1}^J}
\newcommand{\sumkk}{\sum_{k=1}^K}
\newcommand{\prodjj}{\prod_{j=1}^J}
\newcommand{\prodkk}{\prod_{k=1}^K}
\newcommand{\prodll}{\prod_{\ell=1}^L}
\newcommand{\prodjjx}{\prod_{j=1}^{J'}}
\newcommand{\prodkkx}{\prod_{k=1}^{K'}}

\newcommand\set[1]{\ensuremath{\{#1\}}}
\newcommand\bigset[1]{\ensuremath{\bigl\{#1\bigr\}}}
\newcommand\Bigset[1]{\ensuremath{\Bigl\{#1\Bigr\}}}

\newcommand\xpar[1]{(#1)}
\newcommand\bigpar[1]{\bigl(#1\bigr)}
\newcommand\Bigpar[1]{\Bigl(#1\Bigr)}

\newcommand\lrpar[1]{\left(#1\right)}
\newcommand\lrxpar[1]{\!\left(#1\right)}

\newcommand\Bigsqpar[1]{\Bigl[#1\Bigr]}

\newcommand\bigabs[1]{\bigl|#1\bigr|}

\newcommand\lrabs[1]{\left|#1\right|}
\def\rompar(#1){\textup(#1\textup)}    % usage: \rompar(...)
\newcommand\xfrac[2]{#1/#2}

\newcommand\Bigparfrac[2]{\Bigpar{\frac{#1}{#2}}}

\newcommand\expe[1]{e^{#1}}
\def\xexp(#1){e^{#1}}
\newcommand\ceil[1]{\lceil#1\rceil}

\newcommand{\tooo}{\to\infty}
\newcommand{\topoo}{\to+\infty}
\newcommand{\tomoo}{\to-\infty}
\newcommand{\topmoo}{\to\pm\infty}

\newcommand\xtoo{\ensuremath{{x\to\infty}}}

\newcommand\downto{\searrow}
\newcommand\upto{\nearrow}
\newcommand\half{\tfrac12}

\newcommand\iid{i.i.d.\spacefactor=1000}    
\newcommand\ie{i.e.\spacefactor=1000}
\newcommand\eg{e.g.\spacefactor=1000}
\newcommand\viz{viz.\spacefactor=1000}
\newcommand\cf{cf.\spacefactor=1000}
\newcommand\Cf{Cf.\spacefactor=1000}
\newcommand{\as}{a.s.\spacefactor=1000}

\newcommand\ii{\mathrm{i}}

\newcommand{\tend}{\longrightarrow}
\newcommand\dto{\overset{\mathrm{d}}{\tend}}

\newcommand\eqd{\overset{\mathrm{d}}{=}}

\newcommand\bbR{\mathbb R}
\newcommand\bbC{\mathbb C}
\newcommand\bbN{\mathbb N}

\newcommand\bbQ{\mathbb Q}
\newcommand\bbZ{\mathbb Z}
\newcommand\bbZleo{\mathbb Z_{\le0}}
\newcommand\bbZgeo{\mathbb Z_{\ge0}}

\newcounter{CC} 
\newcounter{cc}
\renewcommand\Re{\operatorname{Re}}
\renewcommand\Im{\operatorname{Im}}

\newcommand\E{\operatorname{\mathbb E{}}}
\renewcommand\P{\operatorname{\mathbb P{}}}

\newcommand\Exp{\operatorname{Exp}}

\newcommand\Be{\operatorname{Be}}

\newcommand\Res{\operatorname{Res}}

\newcommand\ga{\alpha}
\newcommand\gb{\beta}
\newcommand\gd{\delta}

\newcommand\gf{\varphi}
\newcommand\gam{\gamma}
\newcommand\gG{\Gamma}
\newcommand\gk{\varkappa}
\newcommand\gl{\lambda}

\newcommand\gs{\sigma}

\newcommand\eps{\varepsilon}

\newcommand\cA{\mathcal A}

\newcommand\cF{\mathcal F}

\newcommand\cM{\mathcal M}

\newcommand\cT{{\mathcal T}}

\newcommand\tF{\widetilde F}
\newcommand\tP{\widetilde P}
\newcommand\tX{\widetilde X}
\newcommand\tY{\widetilde Y}

\newcommand\ett[1]{\boldsymbol1[#1]} 
\newcommand\Bigett[1]{\boldsymbol1\Bigsqpar{#1}} 

\def\[#1]{[\![#1]\!]}

\newcommand\smatrixx[1]{\left(\begin{smallmatrix}#1\end{smallmatrix}\right)}

\newcommand\qq{^{1/2}}
\newcommand\qqw{^{-1/2}}
\newcommand\qw{^{-1}}
\newcommand\qww{^{-2}}
\newcommand\qqc{^{3/2}}

\newcommand\qqq{^{1/3}}
\newcommand\qqqb{^{2/3}}
\newcommand\qqqbw{^{-2/3}}
\newcommand\qqqq{^{1/4}}

\renewcommand{\=}{:=}

\newcommand\intoi{\int_0^1}
\newcommand\intot{\int_0^t}

\newcommand\intotau{\int_0^T}
\newcommand\intoo{\int_0^\infty}
\newcommand\intoooo{\int_{-\infty}^\infty}
\newcommand\intgs{\int_{\gs-\ii\infty}^{\gs+\ii\infty}}
\newcommand\oi{[0,1]}
\newcommand\ooo{(0,\infty)}
\newcommand\oooo{(-\infty,\infty)}

\newcommand\dd{\,\textup{d}}

\newcommand{\mgf}{moment generating function}
\newcommand{\chf}{characteristic function}

\newcommand\lhs{left hand side}
\newcommand\rhs{right hand side}
\newcommand\tpo{_{t\ge0}}

\newcommand\taux{T}
\newcommand\fa{f_{\cA}}

\newcommand{\ogt}{of Gamma type}
\newcommand{\mogt}{moments of Gamma type}
\newcommand{\cx}{a'}
\newcommand{\dx}{b'}
\newcommand{\da}{d}
\newcommand{\Oqw}[1]{O\bigpar{#1\qw}}
\newcommand{\Oqwa}[1]{\Oqw{|#1|}}
\newcommand{\ppi}{\sqrt{2\pi}}
\newcommand{\lpi}{\log\sqrt{2\pi}}
\newcommand{\abs}[1]{|#1|}
\newcommand{\absnuf}[1]{\nu_F(#1)}
\newcommand{\res}[1]{\Res_{#1}}
\newcommand{\xgG}{\Gamma} %Gamma variable
\newcommand{\xgGa}{\xgG_\ga} %Gamma variable
\newcommand{\simin}{\in}
\newcommand{\Bab}{B_{\ga,\gb}}
\newcommand{\xhalf}{_{1/2}}
\newcommand{\xdot}{\cdot}
\newcommand{\Fii}{\,{}_1F_1}
\newcommand{\FF}[2]{\,{}_{#1}F_{#2}}
\newcommand{\xn}{^{(n)}}
\newcommand{\exi}{X_I}
\newcommand{\exii}{X_{II,\ga}}
\newcommand{\exiii}{X_{III,\ga}}
\newcommand{\B}{\mathrm{B}}
\newcommand{\N}{\mathrm{N}}
\newcommand{\U}{\mathrm{U}}
\newcommand{\cfm}{\cF_{\mathsf m}}
\newcommand{\cfo}{\cF_{0}}
\newcommand{\gpu}{generalized \Polya{} urn}
\newcommand{\Levy}{L\'evy}
\newcommand{\Polya}{P\'olya}

\newcommand\urladdrx[1]{{\urladdr{\def~{{\tiny$\sim$}}#1}}}

\begin{document}
\title[Moments of Gamma type]
{Moments of Gamma type and the Brownian supremum process area}

\date{22 February, 2010} %  (typeset \today{} \klockan)} 

\author{Svante Janson}
\address{Department of Mathematics, Uppsala University, PO Box 480,
SE-751~06 Uppsala, Sweden}
\email{svante.janson@math.uu.se}
\urladdrx{http://www.math.uu.se/~svante/}

%\keywords{Brownian motion, supremum process, local time, Brownian areas}
 %%{Primary: <subject>; Secondary: <subject>}

\begin{abstract}
We study positive random variables whose moments can be expressed by
products and quotients of Gamma functions; this includes many standard
distributions. General results are given on existence, series expansion and
asymptotics of density functions. It is shown that the integral of the
supremum process of Brownian motion has moments of this type, 
as well as a related random variable occuring in the study of hashing
with linear displacement,
and the
general results are applied to these variables.
\end{abstract}

\maketitle

\section{Introduction}\label{S:intro}

We say that a positive random variable $X$ has 
\emph{moments of Gamma type} if,
for $s$ in some interval,
\begin{equation}
  \label{gamma}
\E X^s = 
C D^s
\frac{\prodjj \Gamma(a_j s+b_j)}
{\prodkk \Gamma(\cx_k s+\dx_k) } 
\end{equation}
for some integers $J,K\ge0$ and some real 
constants $C$, $D>0$, $a_j$, $b_j$, $\cx_k$, $\dx_k$.
We may and will assume that $a_j\neq 0$ and  $\cx_k\neq 0$ for all $j$
and $k$.
We often denote the \rhs{} of \eqref{gamma} by $F(s)$; this is a
meromorphic function defined for all complex $s$ (except at its poles).

Similarly we say that a real random variable $Y$ has 
\emph{\mgf\ \ogt} if,
for $s$ in some interval,  
\begin{equation}
  \label{gammamgf}
\E e^{sY} = 
C e^{\da s}
\frac{\prodjj \Gamma(a_j s+b_j)}
{\prodkk \Gamma(\cx_k s+\dx_k) } 
\end{equation}
for some integers $J,K\ge0$ and some real 
constants $C$, $\da$, $a_j\neq 0$, $b_j$, $\cx_k\neq 0$, $\dx_k$, 
and that $Y$ has
\emph{\chf\ \ogt} if,
for all real $t$,
\begin{equation}
  \label{gammachf}
\E e^{\ii tY} = 
C e^{\ii t\da}
\frac{\prodjj \Gamma(b_j+\ii a_j t)}
{\prodkk \Gamma(\dx_k +\ii \cx_k t) }
\end{equation}
for some such constants.

Of course, \eqref{gamma} and \eqref{gammamgf} are the same if $X=e^Y$
and $D=e^\da$; further, it will be shown that \eqref{gammamgf} and
\eqref{gammachf} are equivalent by analytic continuation.
Moreover, we shall see that the range of validity of \eqref{gamma},
or \eqref{gammamgf} is always the largest
possible. We summarize these simple but useful observations in
\refT{T1} below.

\begin{remark}\label{Rnonunique}
  The representations in \eqref{gamma}--\eqref{gammachf} are far from
  unique. 
Using the 
duplication formula \eqref{Adouble}, or more generally the
  multiplication formula \eqref{Am}, and
other  relations such as the functional equation $\gG(z+1)=z\gG(z)$, 
a function $F(s)$ 
of this 
form may be rewritten in many different ways.
(More or less  transparently; some equivalent versions may look quite
  different to the unaided eye.) 
See for example Theorems \refand{T2}{TM}.
\end{remark}

\begin{remark}
  \label{RC}
The constant $C$ is determined by the relation
$F(0)=\E X^0=1$, which shows that
\begin{equation}
  \label{gammaC}
C= \frac{\prodkk \Gamma(\dx_k) } 
{\prodjj \Gamma(b_j)}
\end{equation}
provided no $b_j$ or $b'_k$ is a non-positive integer. In general, $C$ can
be found by taking limits as $s\to0$.
\end{remark}

%If $X$ satisfies \eqref{gamma}, we write
%$X\simin \GM((a_1,b_1),\dots,(a_J,b_J);(\cx_1,\dx_1),\dots,(\cx_K,d_K);D)$.

\begin{remark}   \label{RD}
The constant $D$ is just a scale factor:
$X$ satisfies \eqref{gamma} if and only if $X/D$ satisfies the same
equation with $D$ replaced by 1 (\ie, without the factor $D^s$).
Similarly, $Y$ satisfies \eqref{gammamgf} or \eqref{gammachf} if and
only if $Y-\da$ satisfies the same equation with $\da$ replaced by
0. Hence we might assume $D=1$ or $\da=0$ if convenient (but we will
not do so in general). 
\end{remark}

\begin{remark}\label{Rrational}
  If $r\in R$, then $x-r=\Gamma(x-r+1)/\Gamma(x-r)$. Hence, any
  rational function $Q(x)$ with all poles and zeros real can be
  written as a finite product $\prod_\ell \Gamma(x+c_\ell)/\Gamma(x+c'_\ell)$
  with $c_\ell,c'_\ell\in\bbR$. Consequently we may allow such a rational
  factor $Q(s)$ in \eqref{gamma} and \eqref{gammamgf},
or $Q(\ii t)$ in \eqref{gammachf}, without changing the class of distributions.
\end{remark}

\begin{remark}
  \label{Rclosure}
%We make two trivial but useful observations.
If $X$ has moments \ogt{} and $\ga$ is a real number, then
$X^\ga$ has moments \ogt. (Just substitute $\ga s$ for $s$ in \eqref{gamma}.)
Similarly, if $X_1$ and $X_2$ are independent and both have moments of
\ogt, then $X_1X_2$ has too. (Just use $\E (X_1X_2)^s=\E X_1^s\E X_2^s$.)
\end{remark}

Several well-known distributions have moments or \mgf{s} \ogt. We give
a number of examples in \refS{Sex}.

The main motivation for the present paper is that also several less
well-known distributions have moments \ogt. It is then straightforward
to use Mellin transform techniques to obtain expansions or asymptotics
of the density function, and it seems advantageous to do so, and to
study other properties, in general for this class of distributions.

In particular, this paper was inspired by the realization that some
recently studied
random variables have moments \ogt.
One is the integral of the supremum process of a Brownian motion,
\ie, the area under the supremum process (up to some fixed time $T$).
Let $B(t)$, $t\ge0$, be a standard Brownian motion.
Consider the supremum process
$S(t)\=\max_{0\le s\le t} B(t)$, and its integral
\begin{equation}\label{a1}
  \cA(\taux)\=\intotau S(t)\dd t .
%\eqd \intotau L(t)\dd t.
\end{equation}
We further  let $\cA\=\cA(1)$.
For any given $T>0$, 
the usual Brownian scaling
\begin{math}
\set{B(\taux t)}\tpo\eqd\set{\taux\qq B(t)}\tpo
\end{math}
implies the corresponding scaling for the supremum process
$
\set{S(\taux t)}\tpo\eqd\set{\taux\qq S(t)}\tpo
$,
and thus
\begin{equation}\label{a2}
\cA(\taux) =  \taux \intoi S(\taux t)\dd t
\eqd \taux\qqc \cA.
\end{equation}
In particular, $\E \cA(T)^s=T^{3s/2}\E \cA^s$
and it is enough to study $\cA$.

The random area $\cA$ was studied by 
\citet{SJ202}, and using their results
we will in \refS{Sarea} 
prove the following formula, showing that $\cA$ has moments \ogt.
(The result for 
the integer moments $\E\cA^n$, $n\in\bbN$, was given in 
\cite{SJ202}.)
We give several different, but equivalent, formulas of the type
\eqref{gamma} for $\E \cA^s$, which exemplifies \refR{Rnonunique}.
The third version, with only two non-constant Gamma factors is
perhaps the simplest. The last, where all Gamma factors are of the
type $\gG(s/2+b)$ with $0<b\le1$ is of a canonical type where there
are no cancellations of poles, and it is thus easy to see the poles
and zeros, \cf{} \refR{Rcanonical}.

\begin{theorem} \label{T2}
The  moments of $\cA$ are given by,
for $\Re s >-1$,
\begin{equation*}
  \begin{split}
\E \cA^s 
&= \frac{\gG(s+1) \, \gG(s+2/3)}{\gG(2/3) \, \gG(3s/2+1)}
\xdot\Bigpar{\frac{3}{ \sqrt{8}}}^s
\\&
= \frac{\gG(s+1) \, \gG(s+5/3)}{\gG(5/3) \, \gG(3s/2+2)}
\xdot\Bigpar{\frac{3}{ \sqrt{8}}}^s
\\
&= 
%\frac{4\sqrt\pi}{3^{3/2}\gG(2/3) }
\frac{2\gG(1/3) }{3\sqrt\pi}
\xdot
\frac{\gG(3s/2+3/2) }{ \gG(s+4/3)}
\xdot\Bigparfrac {\sqrt{8}}{9}^s
\\
&= \frac{\Gamma(1/3)}{2\qqq\pi}\xdot
\frac{\gG(s/2+1/2)\,\gG(s/2+5/6) }{ \gG(s/2+2/3)}
\xdot\Bigparfrac {2}{3}^{s/2}.
  \end{split}
\end{equation*}
Further, $\E \cA^s =\infty$ for real $s\le-1$.
\end{theorem}

\begin{remark}\label{Rbrown}
  Several related Brownian areas are studied in \citet{SJ201}, for
  example the integral of $|B(t)|$ or the integral of a normalized
  Brownian excursion. These areas do not have moments \ogt. 
In fact, most of the Brownian areas studied there have entire
  functions $\E X^s$ \cite[\S29]{SJ201}, 
which is impossible for moments \ogt, see \refT{T+}\ref{T+:+-}.
(For the remaining two areas in \cite{SJ201}, we have no formal proof
  that they do not have moments \ogt, but it seems very unlikely since
  the integer moments satisfy more  complicated recursion formulas
  \cite{SJ201}.)
\end{remark}

As a consequence of \refT{T2} and our general results in \refS{Sdensity},
we can express the density function of $\cA$ using
the confluent hypergeometric function $\Fii$
(denoted $M$ in \cite{AS} and $\Phi$ in \cite{Lebedev})
or the 
confluent hypergeometric function of the second kind 
$U$ \cite{AS}
(also denoted $\Psi$ \cite{Lebedev}).
(Also the proofs of the next two theorems are given in \refS{Sarea}.)

\begin{theorem}\label{Tarea}
$\cA$ has a density function $\fa(x)$ given by, for $x>0$,
\begin{equation*}
  \begin{split}
f_\cA(x)
\hskip-2em & \hskip2em
=\frac{2\qqqb\Gamma(1/3)}{3\qq\pi}
\sumn (-1)^n
\frac{\gG(n+5/6) }{n!\, \gG(n+2/3)}
\xdot\Bigparfrac {3}{2}^{n+1/2} x^{2n}
\\
&\qquad{}\qquad{}
+
\frac{2\qqqb\Gamma(1/3)}{3\qq\pi}
\sumn(-1)^n
\frac{\gG(n+7/6) }{n!\, \gG(n+4/3)}
\xdot\Bigparfrac {3}{2}^{n+5/6}
x^{2n+2/3}
\\
&=
\frac{2\qq}{\pi\qq}\Fii\lrxpar{\frac56;\frac23;-\frac32 x^2}
+ \frac{2^{-1/6}3\qqq}{\gG(5/6)}\,x^{2/3}
 \Fii\lrxpar{\frac76;\frac43;-\frac32  x^2}
\\
&=
e^{-\frac32  x^2}
\lrpar{
\frac{2\qq}{\pi\qq}\Fii\lrxpar{-\frac16;\frac23;\frac32 x^2}
+ \frac{2^{-1/6}3\qqq}{\gG(5/6)}\,x^{2/3}
\Fii\lrxpar{\frac16;\frac43;\frac32 x^2}
}
\\
&=
\frac{2^{7/6}}{\gG(2/3)}e^{-\frac32x^2} U\lrxpar{-\frac16;\frac23;\frac32 x^2}
\\
&=
\frac{2^{5/6}3\qqq}{\gG(2/3)}x^{2/3} e^{-\frac32x^2} 
U\lrxpar{\frac16;\frac43;\frac32 x^2}
.
  \end{split}
\end{equation*}
\end{theorem}

It follows (most easily from the second formula above) that $f_\cA$
has a finite, positive limit 
$f_\cA(0+)=\sqrt{2/\pi}$ as $x\downto0$. 
More precisely, $f_\cA(x)=\sqrt{2/\pi}+O(x\qqqb)$.

As \xtoo, we obtain from \refT{T2} and our general theorems in \refS{Sdensity}
the following asymptotic result.
Note that
the two terms in the first or second formula for $f_\cA$ in
\refT{T2} are each much larger, of the order
$x^{-5/3}$ by the asymptotics of $\Fii$ in \cite[(13.5.1)]{AS}, but
they cancel each other almost completely for large $x$.

\begin{theorem}\label{Tareaoo}
\begin{equation*}
\fa(x) 
\sim \frac{3^{2/3}\gG(1/3)}{\pi}\, x^{1/3}e^{-3x^2/2}
= \frac{2\cdot 3^{1/6} }{\gG(2/3)}\, x^{1/3}e^{-3x^2/2},
\qquad{x\to\infty}. 
\end{equation*}  
\end{theorem}

This result was conjectured in
\cite{SJ202}, where the weaker result
$\P(\cA>x)=\exp\bigset{-3x^2/2+o(x^2)}$ %,\qquad{x\to\infty}.
was shown from the moment asymptotic 
\begin{equation}
  \label{amom1}
\E \cA^s
\sim
 \frac{\gG(1/3)}{\pi\qq}s^{1/6}\Bigparfrac{s}{3e}^{s/2}
,
\qquad{s\to\infty},
\end{equation}
for integer $s$
and a Tauberian theorem.
(Only integer moments were considered in \cite{SJ202}.
Note that \eqref{amom1} for arbitrary real $s\to\infty$
follows easily from \refT{T2} and Stirling's formula;
see \refT{TRE} and \eqref{amom}.)

\begin{remark}\label{RF20}
\refT{Tareaoo} also follows from any of the last two formulas in
\refT{Tarea} and the 
asymptotic formula for $U$ in \cite[(13.5.2)]{AS}. 
Indeed, this gives an asymptotic expansion with further terms, \cf{}
\refR{Rhighersaddle}; 
in this case, by \cite[(13.5.2)]{AS}, the complete asymptotic
expansion can be written 
\begin{equation*}
\fa(x) 
\sim 
%\frac{3^{2/3}\gG(1/3)}{\pi}\, x^{1/3}e^{-3x^2/2}= 
\frac{2\cdot 3^{1/6} }{\gG(2/3)}\, x^{1/3}e^{-3x^2/2}
\FF20\lrxpar{\frac16,-\frac16;;-\frac23 x\qww},
\qquad \xtoo,
\end{equation*}  
where the hypergeometric series $\FF20$ is divergent and the
asymptotic expansion is interpreted in the usual way:
if we truncate the
series after any fixed number of terms, the error is of the order of
the first omitted term.
(For the general definition of the (generalized) hypergeometric series
$\FF pq$, see \eg{} \cite[Section 5.5]{HyperG}.)
\end{remark}

\refT{Tareaoo} may be compared with similar results for
several other Brownian areas in Janson and Louchard \cite{SJ203}, see also
Janson \cite{SJ201} and \refR{Rbrown}. 
In these results for other Brownian areas, the exponent of $x$ is
always an integer (0, 1 or 2), while here the exponent is $1/3$,
which is related to the power $s^{1/6}$ in 
\eqref{amom1}.

Another example with \mogt{} comes from
\citet{PeterssonI}. He studied the maximum displacement in hashing with linear
  probing, and found for dense tables, after suitable normalization,
convergence to a limit distribution given by a random variable 
$\cM$ with
  the distribution \cite[Theorem 5.1]{PeterssonI}
  \begin{equation}
	\label{np1a}
\P(\cM\le x) = 1-\psi(x\qqc), \qquad x>0,
  \end{equation}
where $\psi(s)\=\E e^ {-s\cA} $ is the Laplace transform of $\cA$.
Equivalently,
  \begin{equation}
	\label{np1b}
\P(\cM> x) = \psi(x\qqc)
=\E e^ {-x\qqc\cA}, \qquad x>0.
  \end{equation}
This type of relation preserves moments \ogt; we give a general result.

\begin{lemma}
  \label{LVZ}
Suppose that\/ $V$ and $Z$ are two positive random variables 
and $\ga>0$.
Then 
\begin{equation}\label{vz1}
  \P(V>x) = \E e^{-x^\ga Z},
\qquad x>0,
\end{equation}
if and only if
\begin{equation}\label{vz2}
  V\eqd T^{1/\ga}/Z^{1/\ga},
\end{equation}
where $T\simin\Exp(1)$ is independent of $Z$.

If \eqref{vz1} or \eqref{vz2} holds, then
\begin{equation}
  \E V^s=\Gamma(s/\ga+1)\E Z^{-s/\ga},
\qquad s>-\ga.
\end{equation}
In particular, if one of $Z$ and $V$ has moments \ogt, then so has the other.
\end{lemma}

We postpone the simple proof until \refS{Shash}. By \eqref{np1b},
\refL{LVZ} applies to $\cM$ and $\cA$, and thus $\cM$ has moments
\ogt. More precisely, \refT{T2} implies the following, see
\refS{Shash} for details.

\begin{theorem}\label{TM}
For $-3/2 < \Re s < 3/2$,
  \begin{equation*}
	\begin{split}
  \E \cM^s 
&= \frac{\gG(1+2s/3)\,\gG(2/3-2s/3)\,\gG(1-2s/3)}{\gG(2/3)\,\gG(1-s)}
\xdot
\Bigparfrac{2}{3\qqqb}^s
%\\&
%= \frac{\gG(1+2s/3)\,\gG(5/3-2s/3)\,\gG(1-2s/3)}{\gG(5/3)\,\gG(2-s)}
%\Bigparfrac{2}{3\qqqb}^s,
\\
&= 
%\frac{4\sqrt\pi}{3^{3/2}\gG(2/3) }\xdot
\frac{2\gG(1/3) }{3\sqrt\pi}
\xdot
\frac{\gG(1+2s/3)\,\gG(3/2-s) }{ \gG(4/3-2s/3)}
\xdot\Bigparfrac {3^{4/3}}{2}^s
\\
&= \frac{\Gamma(1/3)}{2\qqq\pi}\xdot
\frac{\gG(1+2s/3)\,\gG(1/2-s/3)\,\gG(5/6-s/3) }{ \gG(2/3-s/3)}
\xdot\Bigparfrac {3}{2}^{s/3}
\\
&= \frac{\Gamma(1/3)}{2\qqq\pi\qqc}\xdot
\frac{\gG(1/2+s/3)\,\gG(1+s/3)\,\gG(1/2-s/3)\,\gG(5/6-s/3) }{ \gG(2/3-s/3)}
\xdot {6}^{s/3}
.
	\end{split}
  \end{equation*}
Further, $\E \cM^s=\infty$ for real\/ $s\le -3/2$ or $s\ge 3/2$.
\end{theorem}

The special case $s=1$ yields  $\E \cM = 2\Gamma(1/3)/3\qqqb$, as found
by \citet{PeterssonI}. 
\citet{PeterssonI} further proved that $\E \cM^s=\infty$ for  $s\ge2$; we now
see that the sharp threshold is $s=3/2$.

Our general theorems apply again; they show that $\cM$ has a density, and they
yield a series expansion and asymptotics for the density.
(Proofs are given in \refS{Shash}.)
Again, the results can be expressed using various hypergeometric
functions and series. (Again, see \cite{HyperG} for definitions.)

\begin{theorem}
  \label{TMdensity}
$\cM$ has a continuous density function given by, for $x>0$,
\begin{equation*}
  \begin{split}
f_\cM(x)
&= \frac{3\qq\Gamma(1/3)}{2^{5/6}\pi}
\sumn (-1)^n
\frac{\gG(1+n/2)\,\gG(4/3+n/2) }{ \gG(7/6+n/2)\,{n!}}
\Bigparfrac {2}{3}^{n/2}
 x^{3n/2+1/2}
\\
&= \frac{2\qq}{\pi\qq}
x^{1/2}
\FF22\Bigpar{\frac43,1;\frac76,\frac12;\frac{x^3}6 }
- \frac58
x^2
\Fii\Bigpar{\frac{11}6;\frac53;
\frac {x^3}{6}}.
  \end{split}
\end{equation*}
\end{theorem}

In particular, for small $x$ we have the asymptotic formula
\begin{equation}
f_\cM(x)= \frac{2\qq}{\pi\qq}x^{1/2}+O(x^2),
\qquad x\downto0.  
\end{equation}

For large $x$, there is a similar formula, which is the beginning of a
divergent asymptotic expansion (interpreted as in \refR{RF20}):

\begin{theorem}
  \label{TMoo}
As $\xtoo$,
%$f_\cM(x)\sim (3/\sqrt{2\pi})x^{-5/2}$
\begin{equation}\label{laban}
  f_\cM(x)=\frac{3}{\sqrt{2\pi}}x^{-5/2} + O(x^{-7/2}).
\end{equation}
More precisely, $f_\cM(x)$ has as $\xtoo$ an % (divergent) 
asymptotic expansion
\begin{equation}\label{spoke}
  \begin{split}
f_\cM(x)
&\sim
\frac{3}{\sqrt{2\pi}} 
x^{-5/2}\FF31 \Bigpar{\frac32,\frac56,1;\frac23 ;-\frac6{x^3}}	
\\&
\qquad{}
+
\frac{5\,\gG(1/3)}{2\qq3\qqqb\pi\qq}
 x^{-7/2}\FF20 \Bigpar{\frac{11}6,\frac76; ;-\frac6{x^3}}.
  \end{split}
\end{equation}
\end{theorem}

Yet another recent example of \mogt{} comes from the study of
generalized \Polya{} urns \cite{F:exact}, \cite{SJ169}; see \refS{Surn}.

We give some basic reults in \refS{Sbasic}, and further results on
poles and zeros in \refS{Spoles}. Many examples with standard
distributions are given in \refS{Sex}.
Asymptotics of the moments are studied in \refS{Sasymp}, and
asymptotics and series expansions of the density function are given in
\refS{Sdensity}. As said above, we give proofs of the results above
for $\cA$ and $\cM$ in Sections \refand{Sarea}{Shash}, and we give some
results for \gpu{s} in \refS{Surn}.
We end with a couple of more technical examples (counter examples) in
\refS{Sex2} and some further remarks in \refS{Sfurther}.
Some standard formulas for the Gamma function are for convenience
collected
in \refApp{Appa}.

\section{The basic theorem and some notation}\label{Sbasic}

Let $F(s)$ denote the \rhs\ of \eqref{gamma} or
\eqref{gammamgf}. (Thus, the \rhs{} of \eqref{gammachf} is $F(\ii t)$.)
Evidently, $F$ is a meromorphic function in the
complex plane, and all poles are on the real axis. Let $\rho_+$ and
$\rho_-$ be the poles closest to 0:
\begin{equation}\label{rho+-}
  \begin{aligned}
	\rho_+&\=\min\,\set{x>0:x \text{ is a pole of }F} ,
\\
	\rho_-&\=\max\set{x<0:x \text{ is a pole of }F},
  \end{aligned}
\end{equation}
with the interpretation that $\rho_+=\infty$ [$\rho_-=-\infty$] if
there is no pole on $(0,\infty)$ [$(-\infty,0)$].
Thus $-\infty \le \rho_-<0<\rho_+\le\infty$.
Note that we ignore any pole at 0 in the definitions \eqref{rho+-};
however, it follows from \refT{T1} that such a pole cannot exist; $F(s)$
is always analytic at $s=0$.

\begin{theorem}\label{T1}
  Let $X>0$ and $Y$ be random variables connected by $X=e^Y$ and thus
 $Y=\log X$, and let $C$, $D>0$, $\da=\log D$, $a_j\neq0$, $b_j$, $\cx_k\neq0$,
  $\dx_k$ be real constants,
for $j=1,\dots, J\ge0$ and $k=1,\dots,K\ge0$.
Let $F(s)$ be the meromorphic function in \eqref{gamma} and  \eqref{gammamgf}
and let
$\rho_+\in(0,\infty)$ and $\rho_-\in(-\infty,0)$ be defined by
  \eqref{rho+-}. Then the following are equivalent:
  \begin{romenumerate}
\item\label{T1mom0}
\eqref{gamma} holds for all real $s$ in some non-empty interval.	
\item\label{T1mom}
\eqref{gamma} holds for all complex $s$ in the strip $\rho_-<\Re s<\rho_+$.
\item\label{T1mgf0}
\eqref{gammamgf} holds for all real $s$ in some non-empty interval.	
\item\label{T1mgf}
\eqref{gammamgf} holds for all complex $s$ in the strip $\rho_-<\Re s<\rho_+$.
\item\label{T1chf0}
\eqref{gammachf} holds for all real $t\neq0$ in some interval $|t|<t_0$	
with $t_0>0$.
\item\label{T1chf}
\eqref{gammachf} holds for all real $t$.
  \end{romenumerate}
In this case, further $\E X^s=\E e^{sY}=\infty$ if $s\le\rho_-$ or
$s\ge\rho_+$; thus 
\begin{equation*}
\set{s\in\bbR:\E X^s<\infty}=\set{s\in\bbR:\E e^{sY}<\infty}
=(\rho_-,\rho_+).  
\end{equation*}
Equivalently,
\begin{align*}
  \rho_+&=\sup\set{s\ge0:\E X^s<\infty},\\
  \rho_-&=\inf\set{s\le0:\E X^s<\infty}.
\end{align*}
Furthermore, $F(s)=\E X^s=\E e^{sY}\neq0$ when $\rho_-<\Re s<\rho_+$.
\end{theorem}

\begin{proof}%[Proof of \refT{T1}]
\ref{T1mom0}$\iff$\ref{T1mgf0} and \ref{T1mom}$\iff$\ref{T1mgf} are
trivial, since $\E X^s=\E e^{sY}$.
Further, trivially 
\ref{T1mgf}$\implies$\ref{T1mgf0},
\ref{T1chf}$\implies$\ref{T1chf0}
and
\ref{T1mgf}$\implies$\ref{T1chf}.
Hence, to show the equivalences it suffices to show that 
\ref{T1mgf0}$\implies$\ref{T1chf0}
and
\ref{T1chf0}$\implies$\ref{T1mgf}.

\ref{T1chf0}$\implies$\ref{T1mgf}.
%  Suppose that \ref{T1chf0} holds.
Note first that $\gf(t)\=\E e^{\ii tY}\to1$ as $t\to0$. Hence, $F(\ii
t)\to1$ as $t\downto0$, and thus $F$ does not have a pole at
0.

This shows that $F(z)$ is analytic in the strip $\rho_-<\Re z<\rho_+$,
and thus $F(\ii z)$ is analytic in the strip
$-\rho_+<\Im z<-\rho_-$. 
By continuity, $\gf(t)=F(\ii t)$ also for $t=0$ and thus for the
entire interval $(-t_0,t_0)$.
Hence, on this interval at least, $\gf(t)$ equals the boundary values
of the function $F(\ii t)$ which is analytic for $0\le\Im z<-\rho_-$
and by a theorem of \citet{Marcinkiewicz},
$\E e^{-r Y}<\infty$ for every $r\in(0,-\rho_-)$; equivalently, 
$\E e^{r Y}<\infty$ if $\rho_-<r<0$.
By considering $-Y$, we find similarly that 
$\E e^{r Y}<\infty$ if $0<r<\rho_+$.
Consequently,
\begin{equation}\label{jesper}
 \E e^{r Y}<\infty \qquad \text{if $\rho_-<r<\rho_+$}.
\end{equation}
It is well-known that \eqref{jesper} implies that $\psi(z)\=\E e^{zY}$
is defined and finite for $\rho_-<\Re z<\rho_+$ 
and that $\psi(z)$ is an analytic function of $z$ in this strip.
Since
$\psi(\ii t)=\gf(t)=F(\ii t)$ for $|t|<t_0$, analytic continuation yields
$\psi(z)=F(z)$ in this strip, \ie{} \ref{T1mgf} holds.

\ref{T1mgf0}$\implies$\ref{T1chf0}.
Suppose that $\E e^{sY}=F(s)$ for $s\in(a,b)$, with
$-\infty<a<b<\infty$. 
Let $s_0\in(a,b)$ with
$s_0\neq0$ and suppose that $s_0>0$. (The case $s_0<0$ is similar, or
follows by considering $-Y$.)
Thus $\E e^{s_0Y}=F(s_0)<\infty$, and it follows that 
$z\mapsto \E e^{zY}$ is defined and analytic for $0<\Re z< s_0$. Since $\E
e^{zY}=F(z)$  on an interval in this strip, $\E e^{zY}=F(z)$ for
$0<\Re z<s_0$.

For any real $t$, we may take $z=\ii t+\eps$ for $0<\eps<s_0$; letting
$\eps\downto0$ we have $\E e^{(\ii t+\eps)Y}\to\E e^{\ii tY}$ by dominated
convergence (using $\E(1+e^{s_0 Y})<\infty$). 
If further $t\neq0$, then also 
$\E e^{(\ii t+\eps)Y}=F(\ii t+\eps)\to F(\ii t)$
since $F$ has only real poles, and thus $\E e^{\ii tY}=F(\ii t)$.
Hence \ref{T1chf0} holds.

This completes the proof of the equivalences.
Suppose $\E e^{sY}<\infty$ for some $s\ge\rho_+$ (and thus $\rho_+<\infty$).
%$z\mapsto \E e^{zY}$ is defined and analytic for $0<\Re z< s$. 
Letting
$z\upto \rho_+$, we then have, by dominated convergence, $F(z)=\E e^{zY}\to\E
e^{\rho_+Y}<\infty$, while the definition of $\rho_+$ as a pole yields
$F(z)=|F(z)|\to\infty$. This contradiction shows that 
$\E e^{sY}=\infty$ for $s\ge\rho_+$.
Similarly, or by considering $-Y$,
$\E e^{sY}=\infty$ for $s\le\rho_-$.

Finally observe that $F(s)=0$ only when some $a'_ks+b'_k$ is a pole of
$\Gamma$, \ie, a non-positive integer, which implies that $s$ is real.
However, if $\rho_-<s<\rho_+$, then $F(s)=\E e^{sY}>0$.
\end{proof}

\begin{remark}
  The equivalence \ref{T1mgf0}$\iff$\ref{T1mgf} 
(or, equivalently, \ref{T1mom0}$\iff$\ref{T1mom}) is an instance of
  the well-known fact that a (two-sided) Laplace transform of a
  positive function or measure has singularities where the real axis
  intersects the boundary of the natural strip of definition, see
  \eg{} \cite[\S 3.4]{Doetsch}. The result by \citet{Marcinkiewicz}
  used above is a sharper version of this.
\end{remark}
We make some simple but useful observations.

\begin{corollary}
  \label{C1}
The distribution of $X$ is determined by the function $F(s)$ on the
\rhs\ of \eqref{gamma}: If\/ $\E X_1^s=F(s)$ for $s\in I_1$ and 
$\E X_2^s=F(s)$ for $s\in I_2$, for non-empty intervals $I_1$ and
$I_2$, then $X_1\eqd X_2$.
\end{corollary}
\begin{proof}
  By \refT{T1}, 
 $Y_1\=\log X_1$ and  $Y_2\=\log X_2$ have the same \chf{}
$F(\ii t)$. 
Hence, $Y_1\eqd Y_2$ and $X_1\eqd X_2$.
%and the result follows from the uniqeness theorem for \mgf{s}
%\cite[Theorem 4.8.1]{Gut}.
\end{proof}

\begin{remark}\label{RF0}
  As said above, \refT{T1} (\eg, by \ref{T1chf}) implies that $0$ is
not a pole of $F$; furthermore, $F(0)=\E X^0=1$.
\end{remark}

\begin{remark}
Every pole or zero of $F(s)$ must be a pole of one of the $\Gamma$
factors in \eqref{gamma}. However, the converse does not hold, since
poles in the $\gG$ factors may cancel; if $s_0$ is a pole of some
factors in the numerator, but also a pole of at least as many factors
in the denominator, then $s_0$ is a removable singularity of $F$, and $F(s_0)$
is well-defined by continuity. Note that such $s_0$ do not count in
the definition \eqref{rho+-} of $\rho_\pm$.

In particular, $s=0$ may be a pole of some of the Gamma factors in $F(s)$.
(This happens when some $b_j$ or $b'_k$ is 0 or a negative integer. This is the
reason we exclude $t=0$ in \refT{T1}\ref{T1chf0}.) However, by
\refR{RF0},
all such poles must cancel; \ie, there must be an equal number of such
factors in the numerator and denominator in \eqref{gamma}.
\end{remark}

\begin{remark}
  In \refT{T1}\ref{T1chf0}, it is important that we consider an
  interval about 0 (unlike in \ref{T1mom0} and \ref{T1mgf0}).
In fact, for any $\eps>0$, there exist a random variable $Y$ and 
$C, d, a_j,b_j,a_k',b_k'$ such that \eqref{gammachf} holds for
  $|t|\ge\eps$ but not for all $\eps$.

For an example, let $Y$ be any random variable with \chf{}
\ogt, say $\E e^{\ii tY}=F(\ii t)$. 
Further, let $Z$ be a random variable with \chf{} $(1-|t|)_+$,
see
\cite[Section XV.2]{FellerII}, and let $W$ be the mixture of $Z$ and the
constant 0 obtained as $W\=VZ$ with 
$V\simin\Be(1/2)$, \ie{}
$\P(V=0)=\P(V=1)=1/2$, and $V$
independent of $Z$; assume further that $Y$ is independent of $Z$ and $V$.
Then, for $|t|\ge1$, 
the \chf{s}
$\E e^{\ii t Z}=0$ and $\E e^{\ii t W}=\frac12(\E e^{\ii tZ}+1)=\frac12$,
and thus $\tY\=Y+\eps\qw W$ has the \chf
\begin{equation}\label{chffel}
  \E e^{\ii t \tY}
=  \E e^{\ii t Y}
  \E e^{\ii \eps\qw t W}
=\tfrac12 F(\ii t),
\qquad \text{when } |t|\ge\eps.
\end{equation}
Here $\tF(\ii t)\=\frac12 F(\ii t)$ is another function of the type in
\eqref{gammachf}; however \eqref{chffel} does not hold for all $t$
since $\tF(0)=1/2\neq1$.

We do not know whether \eqref{gammachf} for some interval, together
with $F(0)=1$, implies that \eqref{gammachf} holds for all $t$.
\end{remark}

For future use, in particular for the asymptotic results in 
Sections \refand{Sasymp}{Sdensity},
we define the following parameters, given a random variable $X$ or $Y$
or a function $F(s)$ as in \eqref{gamma} or \eqref{gammamgf}:
\begin{align}
  \gamma&\=\sumjj|a_j|-\sumkk|a'_k|,   \label{gam}\\
  \gamma'&\=\sumjj a_j -\sumkk a'_k,   \label{gam'}\\
  \gd&\=\sumjj b_j-\sumkk b'_k-\frac{J-K}2,   \label{gd}\\
  \gk&\=\sumjj a_j\log|a_j|-\sumkk a'_k\log|a'_k|+d,  \label{gk}\\
C_1&\=
C(2\pi)^{(J-K)/2} \frac{\prodjj|a_j|^{b_j-1/2}}{\prodkk|a'_k|^{b'_k-1/2}}.
  \label{CC1}
\end{align}

\begin{proposition}\label{P1}
  The parameters $\gamma,\gamma',\gd,\gk,C_1$ depend on $F$ only and not on
  the particular representation in \eqref{gamma} or \eqref{gammamgf}.
\end{proposition}
The proof is given in \refS{Sasymp}.

\begin{remark}
  \label{R+-}
To replace $X$ by $X\qw$, or equivalently $Y$ by $-Y$, means that
$F(s)$ is replaced by $F(-s)$, which has the same form but with 
$d$ replaced by $-d$ ($D$ by $D\qw$) and similarly
the sign of each $a_j$ and $a'_k$ is changed.
This does not affect $\gamma$, $\gd$,  and $C_1$, but $\gamma'$ and $\gk$ 
change signs.
(This also follows from \eqref{tim1} and \eqref{tim2} below.)
\end{remark}

\begin{remark}
  \label{Rpower}
More generally,  $X^\ga$, with $\ga$ real and non-zero, has parameters
$|\ga|\gam$, $\ga\gam'$, $\gd$, $\ga\gk+\gam'\ga\log|\ga|$, $C_1|\ga|^\gd$.
\end{remark}

\begin{remark}\label{Rproduct}
  If $X=X_1X_2$ with $X_1,X_2$ independent and both having moments
  \ogt, \cf{} \refR{Rclosure}, then the parameters
  $\gam,\gam',\gd,\gk$ for $X$ are the sums of the corresponding
  parameters for $X_1$ and $X_2$, while $C_1$ is the corresponding product.
\end{remark}

\begin{remark}
  \label{Rconj}
If $X$ has moments \ogt, then so has a suitably conjugated (a.k.a.\
tilted) distribution:
If, for simplicity, $X$ has a density function $f(x)$, $x>0$,  let
$\tX$ have the density function $x^rf(x)/\E X^r$ for a real $r$ such
that $\E X^r<\infty$. Then $\E \tX^s=\E X^{s+r}/\E X^r$ and thus $\tX$
has moments \ogt, obtained by a simple substitution in \eqref{gamma}.
It follows that $\gam$, $\gam'$ and $\gk$ are the same for $\tX$ as
for $X$, while $\gd$ is increased by $r\gam'$ and $C_1$ is multiplied
by $e^{r\gk}/\E X^r$. (Cf.\ \eqref{tim2} below.) Clearly, $\rho_+$ and
$\rho_-$ are both decreased by $r$.

For a random variable $Y$ with \mgf{} \ogt, the same applies to $\tY$
with density function $e^{ry}f(y)/\E e^{rY}$, if $Y$ has density
function $f$.
\end{remark}

We note also the following relations, for 
a function $F(s)$ as above: %in \eqref{gamma} or \eqref{gammamgf}:
\begin{lemma}\label{L1} We have
  \begin{align*}
\sum_{j:a_j>0} a_j - \sum_{k:a_k'>0} a_k'
&=\half(\gam+\gam'),
\\
	\sum_{j:a_j<0} |a_j| - \sum_{k:a_k'<0} |a_k'|
&=\half(\gam-\gam').
  \end{align*}
\end{lemma}

\begin{proof}
\begin{align*}
%  \begin{split}
\gam+\gam'
&={\sum_{j} (|a_j|+a_j) - \sum_{k}(|a_k'|+a_k')}
=2\Bigpar{\sum_{j:a_j>0} a_j - \sum_{k:a_k'>0} a_k'}.
\\
\gam-\gam'
&={\sum_{j} (|a_j|-a_j) - \sum_{k}(|a_k'|-a_k')}
=2\Bigpar{\sum_{j:a_j<0} |a_j| - \sum_{k:a_k'<0} |a_k'|}.
%  \end{split}
\qedhere
\end{align*}
\end{proof}

\section{Some examples}\label{Sex}

There are several well-known  examples of distributions with moments
of Gamma type. We collect some of them here. The results below are all
well-known. We usually omit scale parameters that may be added to the
definitions of the distributions.

\begin{exx}{Gamma distribution}\label{Egamma}

Let $\xgG_\ga$ have the Gamma distribution $\gG(\ga)=\gG(\ga,1)$ 
with density function 
$f(x)=x^{\ga-1}e^{-x}/\gG(\ga)$, $x>0$, for a parameter $\ga>0$.
Then, for $\Re s>-\ga$, 
\begin{equation}\label{Gamma}
  \E \xgGa^s = \intoo x^s f(x)\dd x 
= \frac1{\gG(\ga)}\intoo x^{s+\ga-1}e^{-x}\dd x
= \frac{\Gamma(s+\ga)}{\Gamma(\ga)},
\end{equation}
a simple example of moments \ogt.
The \rhs{} of \eqref{Gamma}
is an analytic function in $\Re s>-\ga$, with a pole at
$-\ga$; thus $\rho_+=\infty$, $\rho_-=-\ga$, and 
$\E \xgGa^s=\infty$ when $s\le-\ga$.

We have $\gam=\gam'=1$, $\gd=\ga-1/2$, $\gk=0$, $C_1=\sqrt{2\pi}/\gG(\ga)$.

Note that the different Gamma distributions can be obtained by
conjugation from each other, \cf{} \refR{Rconj}.
\end{exx}

\begin{exx}{Exponential distribution}\label{Eexp}
The exponential distribution $\Exp(1)$ with density function $e^{-x}$, $x>0$, 
is the special case $\ga=1$ of the Gamma distribution in \refE{Egamma}.
We thus obtain from \eqref{Gamma}
(or directly from \eqref{Agamma}), for $T\simin\Exp(1)$,
\begin{equation}\label{Exp}
  \E T^s 
= {\Gamma(s+1)},
\qquad \Re s > -1,
\end{equation}
while $\E T^s=\infty$ when $s\le-1$.

More generally, if $T_\mu$ has an exponential distribution 
$\Exp(\mu)=\gG(1,\mu)$
with mean $\mu$, 
which has the density function is $\mu\qw e^{-x/\mu}$, $x>0$,
then $T_\mu\eqd \mu T$ and 
\begin{equation}\label{Expmu}
  \E T_\mu^s 
= {\Gamma(s+1)}\mu^s,
\qquad \Re s > -1.
\end{equation}
Thus $T_\mu$ has moments \ogt, with
%We have 
$\rho_+=\infty$, $\rho_-=-1$, 
$\gam=\gam'=1$, $\gd=1/2$, $\gk=\log\mu$, $C_1=\sqrt{2\pi}$.
\end{exx}

\begin{exx}{Uniform distribution}\label{EU}

Let $U$ have a uniform distribution $\U(0,1)$ on $\oi$. Then,
obviously, for $\Re s>-1$,
\begin{equation}\label{u1}
  \E U^s = \intoi x^s\dd x=\frac1{s+1}.
\end{equation}
This can be rewritten as a Gamma type formula by the functional
equation \eqref{As+1}, which yields
\begin{equation}\label{u2}
  \E U^s = \frac{\gG(s+1)}{\gG(s+2)},
\qquad \Re s>-1,
\end{equation}
while $\E U^s=\infty$ when $s\le-1$.

Of course, it would be silly to claim that \eqref{u2} is a
simplification of \eqref{u1}, but it shows that the uniform
distribution has moments \ogt{} and thus belongs to the class 
studied here.
% Moreover, it is generalized in the following example.)

We have 
$\rho_+=\infty$, $\rho_-=-1$, 
$\gam=\gam'=0$, $\gd=-1$, $\gk=0$, $C_1=1$.
\end{exx}

\begin{exx}{Beta distribution}\label{Ebeta}

Let $\Bab$ have the Beta distribution $\B(\ga,\gb)$, where $\ga,\gb>0$;
then $\Bab$ has a density function $f(x)=cx^{\ga-1}(1-x)^{\gb-1}$, $0<x<1$,
where $c=\gG(\ga)\qw\gG(\gb)\qw\gG(\ga+\gb)$, \cf{} \eqref{Abeta}.
(Note that $\ga=\gb=1$ yields the uniform distribution in \refE{EU}
as a special case.)
Consequently, by \eqref{Abeta}, for $\Re s>-\ga$,
\begin{equation}
  \label{Beta}
\E \Bab^s
=\intoi x^sf(x)\dd x
=\Bigpar{\frac{\gG(\ga)\gG(\gb)}{\gG(\ga+\gb)}}\qw
\frac{\gG(s+\ga)\gG(\gb)}{\gG(s+\ga+\gb)}
=
\frac{\gG(\ga+\gb)\gG(s+\ga)}{\gG(\ga)\gG(s+\ga+\gb)},
\end{equation}
while $\E \Bab^s=\infty$ if $s\le-\ga$.
We have 
$\rho_+=\infty$, $\rho_-=-\ga$, 
$\gam=\gam'=0$, $\gd=-\gb$, $\gk=0$, $C_1=\gG(\ga+\gb)/\gG(\ga)$.
\end{exx}

\begin{exx}{Chi-square distribution}\label{Echi2}
The chi-square distribution $\chi^2(n)$ is the distribution of 
$Q_n\=\sum_{i=1}^n N_i^2$, where $N_1,N_2,\dots$ are \iid{} standard
normal variables. It is well-known, see \eg{} \cite[Section II.3]{FellerII},
that the chi-square distribution is a Gamma distribution, differing
from the normalized version in \refE{Egamma} by a scale factor; 
more precisely, $Q_n\eqd2\xgG_{n/2}$.  
Consequently, the chi-square distribution has moments \ogt, with, by
\eqref{Gamma},
\begin{equation}\label{chi2}
  \E Q_n^s=2^s\frac{\gG(s+n/2)}{\gG(n/2)},
\qquad \Re s>-n/2,
\end{equation}
while $\E Q_n^s=\infty$ for $s\le -n/2$.
We have $\rho_+=\infty$, $\rho_-=-n/2$, $\gam=\gam'=1$, $\gd=(n-1)/2$,
$\gk=\log2$, $C_1=\sqrt{2\pi}/\gG(n/2)$.

If $n=2$, then $Q_2\eqd 2\xgG_1\eqd 2T\eqd T_2$, which also follows
from \eqref{chi2} and \eqref{Expmu}.
\end{exx}

\begin{exx}{Chi distribution}\label{Echi}
The chi distribution $\chi(n)$ is the distribution of 
$R_n\=\bigpar{\sum_{i=1}^n N_i^2}\qq=Q_n\qq$, with notations as in
\refE{Echi2}. 
Hence, using \eqref{chi2}, the chi distribution has moments \ogt, with
\begin{equation}\label{chi}
  \E R_n^s=\E Q_n^{s/2}=2^{s/2}\frac{\gG(s/2+n/2)}{\gG(n/2)},
\qquad \Re s>-n,
\end{equation}
while $\E R_n^s=\infty$ for $s\le -n$.
We have $\rho_+=\infty$, $\rho_-=-n$, $\gam=\gam'=1/2$, $\gd=(n-1)/2$,
$\gk=0$, $C_1=2^{1-n/2}\pi\qq/\gG(n/2)$, \cf{} \refR{Rpower}.

In particular, the special case $n=1$ shows that if $N\simin \N(0,1)$,
then $|N|\simin\chi(1)$ has moments \ogt{} with
\begin{equation}\label{N}
  \E |N|^s=\pi\qqw2^{s/2}{\gG(s/2+1/2)}.
\end{equation}
\end{exx}

\begin{exx}{$F$-distribution}\label{EF}
  The $F$-distribution is the distribution of 
  \begin{equation}\label{fmndef}
	F_{m,n}\=\frac{Q_m/m}{Q'_n/n}
=\frac nm\frac{Q_m}{Q'_n}, 
  \end{equation}
where $Q_m\simin \chi^2(m)$ and $Q'_n\simin \chi^2(n)$ are independent.
By \eqref{chi2}, this has moments \ogt{} with
\begin{equation}\label{fmn}
  \begin{split}
\E	F_{m,n}^s
&=\Bigparfrac nm^s \E Q_m^s\E (Q'_n)^{-s} 
%\\&
%=\Bigparfrac nm^s (\gG(m/2)\gG(n/2))\qw \gG(s+m/2)\gG(n/2-s),
=\Bigparfrac nm^s\frac{\gG(s+m/2)\gG(n/2-s)}{\gG(m/2)\gG(n/2)},
%\qquad -m/2 < \Re s < n/2.	
  \end{split}
\end{equation}
for $-m/2 < \Re s < n/2$, while $\E F_{m,n}^s=\infty$ for $s\le-m/2$ and
$s\ge n/2$.
We have $\rho_+=n/2$, $\rho_-=-m/2$, $\gam=2$, $\gam'=0$, $\gd=(n+m-2)/2$,
$\gk=\log(n/m)$, 
$C_1=2\pi (\gG(m/2)\gG(n/2))\qw$.
(Cf.\ Remarks \ref{R+-} and \ref{Rproduct}.)
\end{exx}

\begin{exx}{$t$-distribution}\label{ET}
The $t$-distribution is the distribution of 
  \begin{equation}\label{tdef}
	\cT_{n}\=\frac{N}{R_n/\sqrt n},
  \end{equation}
where $N\simin \N(0,1)$ and $R_n\simin \chi(n)$ are independent.
This random variable is not positive; in fact the distribution is symmetric.
However, by \eqref{fmndef} and \eqref{tdef},
$\cT_n^2\eqd F_{1,n}$, so  
if we consider $|\cT_n|\eqd F_{1,n}\qq$, we see from \refE{EF}
that $|\cT_n|$ has moments
\ogt, with 
\begin{equation}\label{t}
  \begin{split}
\E	|\cT_n|^s
=\E	F_{1,n}^{s/2}
= n^{s/2}\frac{\gG(s/2+1/2)\gG(n/2-s/2)}{\sqrt\pi\gG(n/2)},
%\qquad -1 < \Re s < n.	
  \end{split}
\end{equation}
for $-1 < \Re s < n$, while $\E |\cT_{n}|^s=\infty$ for $s\le-1$ and
$s\ge n$.
We have $\rho_+=n$, $\rho_-=-1$, $\gam=1$, $\gam'=0$, $\gd=(n-1)/2$,
$\gk=\frac12\log n$, $C_1=2^{3/2-n/2}\pi\qq/\gG(n/2)$.
(\Cf{} \refR{Rpower}.)
\end{exx}

\begin{exx}{Weibull distribution}\label{EWeibull}
The standard Weibull distribution has the distribution function
\begin{equation}\label{weibull}
  \P(W_\ga\le x) = 1- e^{-x^{\ga}}, \qquad x>0,
\end{equation}
for a parameter $\ga>0$, and thus density function $\ga x^{\ga-1}
e^{-x^{\ga}}$, $x>0$. 

Note that $\ga=1$ yields the exponential distribution in
\refE{Eexp}. Moreover, for any $\ga$ and $y>0$, 
\begin{equation*}
\P(W_\ga^\ga\le y)=\P(W_\ga\le y^{1/\ga}) =1-e^{-y}=\P(T\le y), 
\end{equation*}
where $T\simin\Exp(1)$; hence $W_\ga^\ga\eqd T$ and $W_\ga\eqd T^{1/\ga}$.
It follows from \eqref{Exp} that $W_\ga$ too has moments \ogt, see
\refR{Rclosure}, with
\begin{equation}\label{wa}
  \E W_\ga^s = \E T^{s/\ga}
=\gG(s/\ga+1), \qquad \Re s>-\ga,
\end{equation}
while $\E W_\ga^s=\infty$ if $s\le-\ga$.
We have 
$\rho_+=\infty$, $\rho_-=-\ga$, 
$\gam=\gam'=1/\ga$, $\gd=1/2$, $\gk=\ga\qw\log\ga\qw$, 
$C_1=\sqrt{2\pi/\ga}$, \cf{} \refR{Rpower}.

If $\ga=1$, then obviously $W_1\eqd T\simin\Exp(1)$.
If $\ga=2$, then $W_2\eqd T\qq\eqd 2\qqw Q_2\qq \eqd 2\qqw R_2$,
which also follows by comparing \eqref{wa} and \eqref{chi}.
\end{exx}

\begin{exx}{Stable distribution}\label{Estable}
Let $S_\ga$ be a positive stable random variable with the Laplace
transform $\E e^{-tS_\ga}=e^{-t^\ga}$, with $0<\ga<1$. 
Recall that any positive stable
distribution is of this type, for some $\ga\in(0,1)$, up to a scale
factor, see Feller \cite[Section XIII.6]{FellerII}.
(We may also allow $\ga=1$, but in this exceptional case $S_\ga$ is
degenerate with $S_1=1$ a.s.)

For any $s>0$, by \eqref{Alap},
\begin{equation*}
\Gamma(s)\E S_\ga^{-s}
=  \intoo t^{s-1} \E e^{-tS_\ga}\dd t 
%  \E\intoo t^{s-1} e^{-tS_\ga}\dd t =
=  \intoo t^{s-1}e^{-t^\ga}\dd t,
\end{equation*}
while the change of variables $u=t^\ga$ yields
\begin{equation*}
  \intoo t^{s-1}e^{-t^\ga}\dd t
=\frac1\ga\intoo u^{s/\ga}e^{-u}\frac{\dd u}u
=\ga\qw\Gamma(s/\ga).
\end{equation*}
Hence,
\begin{equation*}
\E S_\ga^{-s} =
  \frac{\Gamma(s/\ga)}{\ga\Gamma(s)}
=  \frac{\Gamma(1+s/\ga)}{\Gamma(1+s)},
\qquad s>0.
\end{equation*}
(In particular, this moment is finite.)
Thus, for $s<0$,
\begin{equation}\label{stab}
\E S_\ga^{s} 
=  \frac{\Gamma(1-s/\ga)}{\Gamma(1-s)}.
\end{equation}
We have shown \eqref{stab} for $s<0$, but \refT{T1} (with $\rho_+=\ga$
and $\rho_-=-\infty$) shows that \eqref{stab} holds whenever $\Re
s<\ga$, while $\E S_\ga^s=\infty$ for $s\ge\ga$. 
(The case $\ga=1$ is exceptional; $\E S_1^s=1$ for every real $s$, so
\eqref{stab} holds but $\rho_+=\infty$.)
Thus $S_\ga$ has moments \ogt.

We have $\gam=\ga\qw-1$, $\gam'=-(\ga\qw-1)$, $\gd=0$,
$\gk=\ga\qw\log\ga$, 
$C_1=\ga\qqw$.
\end{exx}

\begin{exx}{Mittag-Leffler distribution}\label{EML}

The Mittag-Leffler distribution with parameter $\ga\in(0,1)$
can be defined as the distribution of
the random variable $M_\ga\=S_\ga^{-\ga}$, where $S_\ga$ is a positive
stable random variable with $\E e^{-t S_\ga}=e^{-t^\ga}$ as in \refE{Ebeta}.
Since $S_\ga$ has the moments given by \eqref{stab}, the
Mittag-Leffler distribution too has moments \ogt{} given by, \cf{}
\refR{Rclosure}, 
\begin{equation}\label{momML}
  \E M_\ga^s=\E S_\ga^{-\ga s} 
= \frac{\Gamma(s)}{\ga\Gamma(\ga s)}
= \frac{\Gamma(s+1)}{\Gamma(\ga s+1)},
\qquad \Re s>-1,
\end{equation}
while $\E M_\ga^s=\infty$ for $s\le-1$.
In particular, the integer moments are given by 
\begin{equation}\label{momnML}
\E M_\ga^n=\frac{n!}{\gG(n\ga+1)},
\qquad n=0,1,2,\dots
\end{equation}
We have $\rho_+=\infty$, $\rho_-=-1$,
$\gam=\gam'=1-\ga$, $\gd=0$, $\gk=-\ga\log\ga$, $C_1=\ga^{-1/2}$, in
accordance with \refR{Rpower}.

The reason for the name ``Mittag-Leffler distribution'' is that its \mgf{}
is, by \eqref{momnML},
\begin{equation}\label{mgfML}
  \E e^{s M_\ga}=\sumn \E M_\ga^ n \frac{s^n}{n!}
=\sumn  \frac{s^n}{\gG(n\ga +1)},
\end{equation}
which converges for any complex $s$ and is 
known as the
Mittag-Leffler function $E_\ga(s)$ since it was studied by \citet{ML}.
The formula \eqref{mgfML}, or equivalently \eqref{momnML}, is often
taken as the definition of the Mittag-Leffler distribution.

We may also allow $\ga=0$ and $\ga=1$, with $M_0\simin \Exp(1)$ (see
\refE{Eexp}) and $M_1\equiv1$; in both cases \eqref{momML} and
\eqref{mgfML} hold, although for $\ga=0$, \eqref{mgfML} converges only
for $|s|<1$, and $\gd=1/2$, $C_1=\sqrt{2\pi}$, 
while for $\ga=1$, \eqref{momML} trivially holds for all
$s$ with $\E M_\ga^s=1<\infty$ and $\rho_-=-\infty$. 

The Mittag-Leffler distribution was introduced by \citet{Feller49},
see also \cite[Sections XI.5 and XIII.8(b)]{FellerII} and \citet{Pollard}; 
\citet{BM:ML} considered also $\log M_\ga$, which by \eqref{momML} has
\mgf{} $\gG(s+1)/\gG(\ga s+1)$ \ogt, and called its distribution the
``logarithmic Mittag-Leffler distribution''.
Feller \cite{Feller49,FellerII} showed that the Mittag-Leffler
distribution is the limit distribution as $t\to\infty$
of the number of renewals up to time $t$,
properly normalized, of an \iid{} sequence of positive random
variables belonging to the domain of attraction of a stable law.
It emerges also, for example, as the limit distribution of occupancy
times in the 
Darling--Kac theorem, see
Bingham, Goldie, Teugels~\cite[Section 8.11]{BinghamGoldieTeugels}.

In the special case $\ga=1/2$, the duplication formula \eqref{Adouble}
yields
\begin{equation}
  \E M\xhalf^s=\frac{\gG(s+1)}{\gG(s/2+1)}=\pi\qqw 2^s\gG(s/2+1/2),
\end{equation}
which by comparison with \eqref{N} shows that $M\xhalf\eqd \sqrt2|N|$
with $N\simin \N(0,1)$, which is equivalent to the well-known relation
$S\xhalf\eqd \frac12 N\qww$.
\end{exx}

\begin{exx}{A different Mittag-Leffler distribution}\label{EML2}
  Another distribution related to the Mittag-Leffler function
$E_\ga(s)$ in \eqref{mgfML}, and, rather unfortunately,  therefore also called
   ``Mittag-Leffler distribution'' was introduced by \citet{Pillai} as 
  the distribution of a random variable $L_\ga$ with distribution
  function $1-E_\ga(-x^\ga)$, where $0<\ga\le1$; 
equivalently, by \eqref{mgfML},
  \begin{equation}
\P(L_\ga>x)=E_\ga(-x^\ga)
=\E e^{-x^\ga M_\ga}.	
  \end{equation}
This is another instance of the relation \eqref{vz1}, and \refL{LVZ}
shows that $L_\ga$ has moments \ogt{} with, using \eqref{momML}, 
\begin{equation}
  \label{L}
\E L_\ga^s = \gG(s/\ga+1)\E M_\ga^{-s/\ga}
=\frac{\gG(1+s/\ga)\gG(1-s/\ga)}{\gG(1-s)}, 
\qquad -\ga<s<\ga,
\end{equation}
while $\E L_\ga^s=\infty$ if $s\le-\ga$ or (provided $\ga<1$) $s\ge\ga$.
Note also that \refL{LVZ} yields the representation \cite{Pillai}
\begin{equation}   \label{L2}
L_\ga\eqd
T^{1/\ga} M_\ga^{-1/\ga}
= T^{1/\ga}S_\ga,
\end{equation}
with $T\simin\Exp(1)$ and the stable variable $S_\ga$ independent.
Equivalently, see \refE{EWeibull},
$L_{\ga}\eqd W_\ga S_\ga$, with the Weibull variable
$W_\ga$ and $S_\ga$ independent. 
It follows easily from \eqref{L2} that $L_\ga$ has the Laplace
transform $\E e^{-t L_\ga}=(1+t^\ga)\qw$, $t>0$ \cite{Pillai}.

For $\ga=1$, we have $L_1\simin \Exp(1)$.
For $0<\ga<1$, $\rho_+=\ga$, $\rho_-=-\ga$, $\gam=2/\ga-1$, $\gam'=1$,
$\gd=1/2$, $\gk=0$, $C_1=\sqrt{2\pi}/\ga$. 
(For example by \refR{Rproduct}.)
\end{exx}

\begin{exx}{Pareto distribution}\label{EPareto}
The Pareto($\ga$) distribution, where $\ga>0$, is the distribution of
a random variable $P_\ga$ with $\P(P_\ga>x)=x^{-\ga}$, $x\ge1$.
Hence $P_\ga$ has density function $\ga x^{-\ga-1}$, $x>1$.
Direct integration shows that the moments are \ogt{} and given by
\begin{equation}\label{p1}
  \begin{split}
\E P_\ga^s=\int_1^{\infty}\ga x^{s-\ga-1}=\frac{\ga}{\ga-s}	
=\frac{\ga\gG(\ga-s)}{\gG(\ga-s+1)},
\qquad \Re s<\ga.	
  \end{split}
\end{equation}
Hence $\rho_+=\ga$, $\rho_-=-\infty$, 
$\gam=\gam'=0$, $\gd=-1$, $\gk=0$, $C_1=\ga$.

We have $P_\ga\eqd U^{-1/\ga}$ with $U\simin \U(0,1)$, so alternatively
these result follow from \refE{EU} and Remarks \refand{Rclosure}{Rpower}.
\end{exx}

\begin{exx}{Shifted Pareto distribution}\label{EPareto2}
The shifted Pareto variable $\tP_\ga\=P_\ga-1$, where $\ga>0$, 
has support $(0,\infty)$ and density function $\ga (x+1)^{-\ga-1}$, $x>0$.
The moments are given by, using \eqref{Abeta2},
\begin{equation}\label{p2}
  \begin{split}
\E \tP_\ga^s= \int_1^{\infty}\ga x^s(x+1)^{-\ga-1}
=\ga\frac{\gG(s+1)\gG(\ga-s)}{\gG(\ga+1)}
=\frac{\gG(s+1)\gG(\ga-s)}{\gG(\ga)},
%\qquad \Re s<\ga.	
  \end{split}
\end{equation}
for $-1<\Re s<\ga$, while $\E \tP_\ga^s=\infty$ if $s\le-1$ or $s\ge\ga$.
Hence also the shifted Pareto distribution has moments \ogt, with
$\rho_+=\ga$, $\rho_-=-1$, 
$\gam=2$, $\gam'=0$, $\gd=\ga$, $\gk=0$, $C_1=2\pi/\gG(\ga)$.

For $\ga=1$, \eqref{p2} yields $\E\tP_1^s=\gG(1+s)\gG(1-s)=\E\tP_1^{-s}$;
hence $\tP_1\eqd\tP_1\qw$ (so $\log\tP_1$ has a symmetric
distribution).
Using \eqref{Asin}, we have
\begin{equation*}
  \E \tP_1^s=\gG(1+s)\gG(1-s)=s\gG(s)\gG(1-s)=\frac{\pi s}{\sin(\pi s)},
\qquad -1<\Re s<1.
\end{equation*}
Further, comparing \eqref{p2} and \eqref{fmn} we see that $\tP_1\eqd F_{2,2}$.

More generally, \eqref{p2} and \eqref{Gamma} imply that if
$\gG_1\simin\gG(1)=\Exp(1)$ and $\gG_\ga'\simin\gG(\ga)$ are
independent, then 
$\E(\gG_1/\gG'_\ga)^s=\E\gG_1^ s\E(\gG'_\ga)^{-s}=\E\tP_\ga^s$ (for
suitable $s$); thus $\tP_\ga\eqd\gG_1/\gG'_\ga$ by \refC{C1}.
This is also an instance of \refL{LVZ} (with $\ga=1$ in \eqref{vz1}
and \eqref{vz2}).
\end{exx}

\begin{exx}{Extreme value distributions}\label{Eextreme}
There are three types of extreme value distributions, see \eg{}
\citetq{Chapter I}{LLR}.
We let $\exi,\exii,\exiii$ denote corresponding random variables;
they have the distribution functions
\begin{align}
  \P(\exi\le x)&=e^{-e^{-x}}, 
& -\infty<x&<\infty, \label{exi}
\\
  \P(\exii\le x)&=e^{-x^{-\ga}}, 
& 0<x&<\infty, \label{exii}
\\
\qquad\qquad
  \P(\exiii\le x)&=e^{-(-x)^{\ga}}, 
& -\infty<x&<0,\qquad\qquad \label{exiii}
\end{align}
  where (for types II and III) $\ga>0$ is a real parameter.

The distribution \eqref{exi} (the Gumbel distribution)
has the entire real line as support, and is therefore not qualified to
have \mogt. It has, however, \mgf\ \ogt, see \refE{EGumbel}

The distribution \eqref{exii} (the Fr\'echet distribution)
satisfies, with $T\simin\Exp(1)$, 
\begin{equation*}
 \P(\exii\le x)=\P(T\ge x^{-\ga})=\P(T^{-1/\ga}\le x),
\qquad x>0,
\end{equation*}
and thus $\exii\eqd T^{-1/\ga}\eqd1/W_\ga$, see \refE{EWeibull}. 
Hence it has \mogt{} with, see \eqref{Exp} and \eqref{wa},
\begin{equation}
  \E\exii^s=\gG(1-s/\ga),
\qquad \Re s<\ga.
\end{equation}
We have 
$\rho_+=\ga$, $\rho_-=-\infty$, 
$\gam=1/\ga$, $\gam'=-1/\ga$, 
$\gd=1/2$, $\gk=\ga\qw\log\ga$, 
$C_1=\sqrt{2\pi/\ga}$, 
\cf{} Remarks \refand{R+-}{Rpower} and Examples
\refand{Eexp}{EWeibull}. 

The distribution \eqref{exiii}
satisfies $\P(-\exiii\ge x)=e^{-x^\ga}$, $x>0$, so
  $-\exiii\eqd W_\ga$ and $\exiii\eqd-W_\ga$, see
  \refE{EWeibull}.
By \eqref{wa}, $|\exiii|=-\exiii$ has \mogt{} with
\begin{equation}
  \E|\exiii|^s=\gG(1+s/\ga),
\qquad \Re s>-\ga.
\end{equation}
We have 
$\rho_+=\infty$, $\rho_-=-\ga$, 
$\gam=\gam'=1/\ga$, $\gd=1/2$, $\gk=\ga\qw\log\ga\qw$, 
$C_1=\sqrt{2\pi/\ga}$.
\end{exx}

We have so far considered distributions with moments \ogt; we now turn
to a few examples with \mgf{} \ogt. Of course, if $X$ is any of the
examples above, $\log X$ yields such an example.
(One such example was mentioned in \refE{EML}.)

\begin{exx}{Exponential distribution again}\label{Eexpmgf}
We noted in \refE{Eexp} that $T\simin\Exp(1)$ has moments \ogt. It has
\mgf{} \ogt{} too, since
\begin{equation}\label{expmgf}
  \E e^{sT}=\intoo e^{sx-x}\dd x =\frac{1}{1-s}=\frac{\gG(1-s)}{\gG(2-s)},
  \qquad \Re s< 1.
\end{equation}
In fact, $T\eqd -\log U\eqd \log P_1$, and equivalently $U\eqd e^{-T}$
and $P_1\eqd e^{T}$; compare \eqref{expmgf} to \eqref{u2} and \eqref{p1}.    
We have $\rho_+=1$, $\rho_-=-\infty$, 
$\gam=\gam'=0$, $\gd=-1$, $\gk=0$, $C_1=1$.
\end{exx}

\begin{exx}{Gamma distribution again}\label{Egammamgf}
A Gamma distributed random variable $\xgG_n\simin\gG(n)$ with integer
$n\ge1$ can be obtained by 
taking the sum of $n$ independent copies of $T\simin\Exp(1)$, and thus
\eqref{expmgf} implies that $\xgG_n$ has \mgf{} \ogt{} with
\begin{equation}\label{Gammamgf}
  \E e^{s\xgG_n}=\lrpar{\E e^{sT}}^n
=\frac1{(1-s)^n}
=\frac{\gG(1-s)^n}{\gG(2-s)^n},
  \qquad \Re s< 1.
\end{equation}
We have $\rho_+=1$, $\rho_-=-\infty$, 
$\gam=\gam'=0$, $\gd=-n$, $\gk=0$, $C_1=1$.

More generally, for any real $\ga>0$, $\xgGa$ has \mgf{} 
$\E e^{s\xgGa}=(1-s)^{-\ga}$. If $\ga$ is not an integer, then this
function has a singularity as $s\upto1$ that is not a pole; hence $\E
e^{s\xgGa}$ cannot be extended to a meromorphic function in the
complex plane, and $\xgGa$ does \emph{not} have \mgf{} \ogt.
Consequently, the Gamma distribution $\gG(\ga)$ has \mgf{} \ogt{} if
and only if $\ga$ is an integer.
\end{exx}

\begin{exx}{Gumbel distribution}\label{EGumbel}
If $\exi$ has the Gumbel distribution \eqref{exi}, and
$T\simin\Exp(1)$, then  
\begin{equation*}
  \P\bigpar{e^{-\exi}\ge x}
=\P(\exi\le-\log x) = e^{-x}=\P(T\ge x),
\qquad x>0,
\end{equation*}
so $e^{-\exi}\eqd T$ and $\exi\eqd-\log T$.
Consequently, the Gumbel distribution has \mgf\ \ogt{} with, see \eqref{Exp},
\begin{equation}
  \E e^{s\exi}=\E T^{-s}=\gG(1-s),
\qquad \Re s<1.
\end{equation}
We have 
$\rho_+=1$, $\rho_-=-\infty$, 
$\gam=1$, $\gam'=-1$, 
$\gd=1/2$, $\gk=0$,
$C_1=\sqrt{2\pi}$.
\end{exx}

\begin{exx}{L\'evy area}\label{ELevy}

The L\'evy stochastic area is defined by the stochastic integral
$A_t\=\intot X_u\dd Y_u-\intot Y_u\dd X_u$, where 
$(X_u,Y_u)$, $u\ge0$, is a two-dimensional Brownian motion starting at
0 (\ie,
$X_u$ and $Y_u$, $u\ge0$, are two independent standard Brownian
motions).
By Brownian scaling, $A_t\eqd tA_1$, so we consider only $A\=A_1$.
Then, for real $t$, 
see \eg{} \citetq{Theorem II.43}{Protter}, using \eqref{Asin},
\begin{equation}
  \E e^{\ii tA}
=\frac1{\cosh\xpar{t}}
=\frac1{\cos\xpar{\ii t}}
=\frac1{\sin(\frac\pi2+\ii t)}
=\frac{\gG(\frac12+\frac{\ii t}\pi)\gG(\frac12-\frac{\ii t}\pi)}{\pi}
.
\end{equation}
Consequently, by \refT{T1}, $A$ has \mgf{} \ogt, with
$\rho_\pm=\pm\pi/2$ and
\begin{equation}\label{levymgf}
  \E e^{s A}
=\frac1{\cos s}
=\frac{\gG(\frac12+\frac s\pi)\gG(\frac12-\frac s\pi)}{\pi},
\qquad |\Re s|<\pi/2.
\end{equation}

We have $\gam=2/\pi$, $\gam'=0$, $\gd=0$, $\gk=0$, $C_1=2$.

It is known that $A$ has the density function $1/2\cosh(\pi x/2)$,
$-\infty<x<\infty$, see 
\eg{} \citetq{Corollary to  Theorem II.43}{Protter}
and \refE{ElevyD} below.

Of course, $e^A$ has moments \ogt, and so has $e^{cA}$ for every real
$c$. A comparison with \eqref{fmn} shows that
$e^{\pi A}\eqd F_{1,1}$.
Hence, $A\eqd\pi\qw\log F_{1,1}$.
\end{exx}

\begin{remark}
  \label{Rnot}
We have shown that a large number of classical continuous
distributions have \mogt, but there are exceptions. For example, if
$X\simin \U(1,2)$, then $\E X^s=(2^{s+1}-1)/(s+1)$ has complex zeros at
$-1+2\pi\ii k/\log 2$, $k=\pm1,\pm2,\dots$, which is impossible when
\eqref{gamma} holds. More generally, if $X$ is non-degenerate and is
supported in a finite interval $[a,b]$ with $0<a<b<\infty$, then $\E
X^s$ is an entire function of $s$, which by \refT{T+} below shows that
$X$ does not have \mogt.
\end{remark}

\section{Poles and zeros}\label{Spoles}

In this section it will be convenient to consider the class $\cF$ of
all functions of the type in \eqref{gamma} and \eqref{gammamgf},
regardless of whether they equal $\E X^s$  [$\E e^{sY}$] for some
random variable $X$  [$Y$] or not. Thus, $\cF$ is the set of functions
\begin{equation}
  \label{gammaF}
F(s) = 
C D^s
\frac{\prodjj \Gamma(a_j s+b_j)}
{\prodkk \Gamma(\cx_k s+\dx_k) }, 
\end{equation}
where $J,K\ge0$ and $C$, $D$, $a_j$, $b_j$, $\cx_k$, $\dx_k$ are real
with $D>0$ and $a_j\neq 0$, $\cx_k\neq 0$ for all $j$ and $k$.
We let $\cfm\subset\cF$ be the set of such functions that appear in
\eqref{gamma} and \eqref{gammamgf}, 
\ie, the set of $F(s)\in\cF$ such that $F(s)=\E X^s$
for some positive random variable $X$ and $s$ in some interval.

A function $F(s)\in\cF$ 
% on the \rhs{} of \eqref{gamma} and \eqref{gammamgf} 
is a meromorphic function of $s$ in the complex plane
$\bbC$. 
We can easily locate its poles (and zeros) precisely as follows.
Define $\nu(s)=\nu_F(s)$ for all $s\in\bbC$ by
\begin{equation}\label{nuf}
  \nu_F(s)=
  \begin{cases}
	m& \text{if $F$ has a pole of order $m$ at $s$},\\
	-m& \text{if $F$ has a zero of order $m$ at $s$},\\
	0& \text{otherwise (\ie, $F$ is regular at $s$ with $F(s)\neq0$)}.
  \end{cases}
\end{equation}
(For $F\in\cfm$ with $F(s)$ an extension of $\E X^s$, we also write $\nu_X$.)

Since $D^s$ has neither poles nor zeros, while $\gG(z)$ has a simple
pole at each $z\in\bbZleo\=\set{0,-1,-2,\dots}$ but no zeros,
\begin{equation}\label{nuf2}
  \begin{split}
  \nu_F(s)
&=\sumjj\ett{a_js+b_j\in\bbZleo}-\sumkk\ett{a'_ks+b'_k\in\bbZleo}
\\&
=\sumjj\Bigett{s\in\Bigset{-\frac{n+b_j}{a_j}:n\in\bbZgeo}}
-\sumkk\Bigett{s\in\Bigset{-\frac{n+b'_k}{a'_k}:n\in\bbZgeo}}.
  \end{split}
\end{equation}
Note that all poles and zeros of $F$ lie on the real axis.

We further define $\cfo\=\set{F\in\cF:\nu_F(0)=0}$, \ie, the set of
functions $F\in\cF$ that have neither a pole nor a zero at 0. By
\refR{RF0}, $\cfm\subset\cfo\subset\cF$. Note that $\cF$ is a group
under multiplication and that $\cfo$ is a subgroup.

For $F\in\cfm$ we defined $\rho_\pm$ in \refS{Sbasic}; 
the definition \eqref{rho+-} can be written
\begin{equation}\label{rhofm}
  \begin{aligned}
	\rho_+
&\=\min\set{x>0:\nu_F(x)>0} ,
\\
	\rho_-
&\=
\max\set{x<0:\nu_F(x)>0}.
  \end{aligned}
\end{equation}
For general $F\in\cF$ we define 
\begin{equation}\label{rhof}
  \begin{aligned}
	\rho_+
&\= \min\set{x\ge0:\nu_F(x)\neq0} ,
\\
	\rho_-
&\= \max\set{x\le0:\nu_F(x)\neq0},
  \end{aligned}
\end{equation}
and note that this is consistent with \eqref{rhofm} for $F\in\cfm$ 
by the fact (\refT{T1}) that such $F$ has no zeros in
$(\rho_-,\rho_+)$ and no pole at 0.

If all $a_j,a'_k,b_j,b'_k>0$, 
then $F(s)$ has no poles or zeros in
the right halfplane $\Re s\ge0$, since none of the factors in
\eqref{gammaF} has; thus $\rho_+=\infty$. 
Similarly, if  all 
$a_j,a'_k<0$ and
$b_j,b'_k>0$, 
then $F(s)$ has no poles or zeros in
the left halfplane $\Re s\le0$ and $\rho_-=-\infty$.
The converses do not hold, because the representation \eqref{gammaF} is
not unique and we may, \eg, add cancelling factors that separately
have poles at other places. However, we may always choose a
representation of the desired type; this is part of the following theorem.

\begin{theorem}
  \label{T+}
Let $F\in\cfo$ and 
let $\rho_\pm$ be as in \eqref{rhof}.
\begin{romenumerate}
  \item\label{T+:}
It is always possible to find a representation \eqref{gammaF} of\/ $F(s)$
with all $b_j,b'_k>0$.
  \item\label{T+:-}
If\/ $\rho_-=-\infty$, then $F(s)$ has a representation \eqref{gammaF}
with all 
$a_j,a'_k<0$ and
$b_j,b'_k>0$.
  \item\label{T+:+}
If\/ $\rho_+=\infty$, then $F(s)$ has a representation \eqref{gammaF}
with all $a_j,a'_k,b_j,\allowbreak b'_k>0$.
  \item\label{T+:+-}
If $\rho_+=\infty$ and  $\rho_-=-\infty$
(\ie, $F(s)$ is entire and without zeros),
then $F(s)=CD^s$ for some real constants $C$ and $D>0$. %$d$. 
%except in the trivial case $X=D$ \as{} for some $D>0$ and thus $\E X^s=D^s$.
\end{romenumerate}

In particular, if $X$ is a random variable with moments \ogt, 
this applies to the meromorphic extension $F(s)$ of $\E X^s$.
In this case, 
it is not possible that both $\rho_+=\infty$ and  $\rho_-=-\infty$
(\ie, that\/ $\E X^s$ is entire), 
except in the trivial case $X=D$ \as{} for some $D>0$ (and thus $\E X^s=D^s$).
\end{theorem}

We begin by proving a lemma.

\begin{lemma}
  \label{L+}
Suppose that $a_j,a'_j,b_k,b'_k$ are real with $a_j,a'_k>0$ and that 
$F(s)\=\prodjj\gG(a_js+b_j)/\prodkk\gG(a_k's+b_k')$ is analytic and
non-zero in a half-plane $\Re s<B$ for some $B\in(-\infty,\infty)$.
Then $F(s)=e^{\ga s} Q(s)$ for some rational function $Q$ with real
poles and zeros and some $\ga\in \bbR$.
\end{lemma}

\begin{proof}
Say that two non-zero real numbers $a$ and $a'$ are
  \emph{commensurable} if $a/a'\in\bbQ$; this is an equivalence
  relation on $\bbR^*\=\bbR\setminus\set0$.
(In algebraic language, the equivalence classes are the cosets of
  $\bbQ^*$ in $\bbR^*$.) 
The poles of $\gG(as+b)$ are regularly spaced with distances $1/|a|$.
Thus, if $a$ and $a'$ are incommensurable, and $b,b'$ are any real  numbers,
then $\gG(as+b)$ and $\gG(a's+b')$ have at most one common pole.

We divide the set $\set{a_j}_{j=1}^J\cup\set{a'_k}_{k=1}^K$ into
equivalence classes of commensurable numbers. This gives a
corresponding factorization $F(s)=\prodll F_\ell(s)$ where 
$L$ is the number of equivalence classes and each $F_\ell$ is of the
same form as $F$ but with all $a_j$ and $a'_k$ commensurable.
It follows that two different factors $F_{\ell_1}(s)$ and
$F_{\ell_2}(s)$ have at most a finite number of common zeros or poles;
hence, by decreasing $B$, we may assume that there are no such common
poles or zeros with $\Re s<B$. Hence, a pole or zero of a
factor $F_\ell(s)$ in $\Re s<B$ cannot be cancelled by another factor,
and thus such poles or zeros do not exist. Consequently, each factor
$F_\ell(s)$ satisfies the assumption of the lemma, and we may thus treat
each $F_\ell(s)$ separately. This means that we may assume that all $a_j$
and $a'_k$ are commensurable.

In this case, there is a positive real number $r$ such that all $a_j$
and $a'_k$ are (positive) integer multiples of $r$. Using Gauss's
multiplication formula \eqref{Am}, we may convert each factor
$\gG(a_js+b_j)$ or $\gG(a_k's+b_k')$ into a product of a constant, an
exponential factor $e^{\gb s}$, and a number of Gamma factors
$\gG(rs+u_i)$ with $u_i\in\bbR$. Using the functional relation
\eqref{As+1}, we may further assume that each $u_i\in(0,1]$, provided
  we also allow factors $(rs+u)^{\pm1}$ with $u\in\bbR$.
Collecting the factors, we see that
\begin{equation}
  \label{Fkill}
F(s)=e^{\ga s}Q(s)\frac{\prodjjx\gG(rs+u_j)}{\prodkkx\gG(rs+u_k')}
\end{equation}
for a constant $\ga$, a rational function $Q(s)$ with real zeros
  and poles, and some $u_j,u'_k\in(0,1]$. 
We may assume that $Q$ has no poles or zeros with $\Re s<B$ (by again
  decreasing $B$ if necessary). The factors $\gG(rs+u_j)$ and
  $\gG(rs+u_k')$ have no zeros but each has an infinite number of
  poles with $\Re s<B$, and two factors $\gG(rs+u_j)$ and
  $\gG(rs+u_k')$ have disjoint sets of poles unless $u_j=u_k'$.
Since the poles with $\Re s<B$ must cancel in \eqref{Fkill}, this
  shows that the Gamma factors must cancel each other completely, and
  thus \eqref{Fkill} reduces to $F(s)=e^{\ga s}Q(s)$.
\end{proof}

\begin{proof}[Proof of \refT{T+}]
\ref{T+:}:
Using $\gG(z)=\gG(z+1)/z$ repeatedly on any term with $b_j\le0$ or
$b'_k\le0$, we may write $F(s)$ as
\begin{equation}
  \label{erika}
C D^s Q(s)
\frac{\prodjj \Gamma(a_j s+\tilde b_j)}
{\prodkk \Gamma(\cx_k s+\tilde b_k') }, 
\end{equation}
with $\tilde b_j,\tilde b_k'>0$, where $Q(s)$ is a rational function
with only real poles and zeros;
$Q(s)=c\prod_i(s-r_i)/\prod_\ell(s-r'_\ell)$ for some real $r_i$ and
$r'_\ell$, $i=1,\dots,I$ and $\ell=1,\dots,L$, say; we may further
assume that $r_i\neq r'_\ell$ for all $i$ and $\ell$.
Since 0 is not a pole or zero of $F$, it is by \eqref{erika}
not a zero or pole of $Q$,
and thus  all $r_i,r'_\ell\neq0$. 

If $r<0$, then 
\begin{equation}
  s-r=s+|r|=\frac{\gG(s+|r|+1)}{\gG(s+|r|)},
\end{equation}
and if $r>0$, then 
\begin{equation}\label{s-r}
  s-r=-(r-s)=-\frac{\gG(-s+r+1)}{\gG(-s+r)},
\end{equation}
so $Q(s)$ may be written as product of quotients of Gamma factors of
the desired type (and a constant), and thus the result follows from
\eqref{erika}. 

\ref{T+:-}:
We use \ref{T+:} and may thus assume that \eqref{gammaF} holds with
$b_j,b'_k>0$. We factorize $F(s)$ as, with $d=\log D$, 
\begin{equation}
  \label{ffact}
F(s)=C e^{ds} F_+(s)F_-(s),
\end{equation}
where $F_+(s)$ contains all factors $\gG(a_js+b_j)$ and
$\gG(a'_ks+b'_k)$ with $a_j,a'_k>0$, and 
$F_-(s)$ contains the factors 
%$\gG(a_js+b_j)$ and $\gG(a'_ks+b'_k)$ 
with $a_j,a'_k<0$,
so $F_-$ is of the desired form.

Since $b_j,b'_k>0$, $F_-(s)$ has no poles or zeros with $\Re s\le 0$.
By assumption, $\rho_-=-\infty$, and thus $F(s)$ is analytic and
non-zero in the half-plane $\Re s\le 0$. Consequently, by
\eqref{ffact}, $F_+(s)$ also has no poles or zeros in $\Re s\le0$.
By \refL{L+},
$F_+(s)=Q(s)e^{\ga s}$ for a rational function
$Q(s)=C_1\prod_j(s-u_j)/\prod_k(s-v_k)$ where $u_j$ and $v_k$ are real.
We may assume that $u_j\neq v_k$ for all $j,k$, and then each $u_j$ or
$v_k$ is a zero or pole of $F_+(s)$, and thus all $u_j,v_k>0$.
Using \eqref{s-r}, we thus can write $Q(s)$ in the desired form; hence
$F_+(s)$ and $F(s)$ can be so written.

\ref{T+:+}:
Follows from \ref{T+:-} by replacing  $F(s)$ by $F(-s)$.

\ref{T+:+-}:
In this case, $F(s)$ is an entire function without zeros. We use again
the factorization \eqref{ffact}. By the proof of \ref{T+:-},
$F_+(s)=e^{\ga s}Q(s)$ for a rational function $Q$.
By symmetry, as in the proof of \ref{T+:+},
similarly $F_-(s)=e^{\ga_- s}Q_-(s)$ for another rational function $Q_-$.
Thus,
\begin{equation}
  F(s)=C e^{(d+\ga+\ga_-)s}Q(s)Q_-(s).
\end{equation}
Hence, the rational function $Q(s)Q_-(s)$ is entire and has neither
poles nor zeros; consequently, it is constant.
Thus $F(s)=C_1 e^{d_1s}$ for some $C_1$ and $d_1$. 
%Further, $C_1=F(0)=1$,and thus $F(s)=D_1^s$ with $D_1\=e^{d_1}$.
\end{proof}

As a consequence we show that the function $\nu_F(s)$ describing the
poles and zeros of $F(s)$ essentially determines $F$, and thus the
distribution of $X$ and $Y$ satisfying \eqref{gamma} or \eqref{gammamgf}.

\begin{theorem}\label{Tunique}
  If $F_1,F_2\in\cF$ and $\nu_{F_1}=\nu_{F_2}$, then $F_2(s)=cD^s
  F_1(s)$ for some real constants $C\neq0$ and $D>0$.
\end{theorem}

\begin{proof}
  $F\=F_2/F_1\in\cF$ and $\nu_F(s)=\nu_{F_1}(s)-\nu_{F_2}(s)=0$ for
  every $s$. Hence, $F\in\cfo$ and $\rho_-=-\infty$, $\rho_+=\infty$;
  thus \refT{T+}\ref{T+:+-} shows that $F(s)=CD^s$.
\end{proof}

\begin{corollary}\label{Cunique}
  If\/ $X_1$ and $X_2$ are positive random variables with \mogt{} and 
$\nu_{X_1}=\nu_{X_2}$, then $X_2\eqd DX_1$ for some constant $D>0$.
In other words, 
a distribution with \mogt{} is
uniquely determined up to a scaling factor by the function $\nu_X(s)$.

Similarly, if $Y_1$ and $Y_2$ have  \mgf{s}  \ogt{} with the same
$\nu$, then $Y_2\eqd Y_1+d$ for some real constant $d$.
\end{corollary}

\begin{proof}
  Let $F_j(s)$ be the meromorphic extension of $\E X_j^s$,
  $j=0,1$. Then $F_2(s)=CD^sF_1(s)$ by \refT{Tunique}. Setting $s=0$
  we find $C=1$, and thus, for $s$ in some interval, $\E
  X_2^s=F_2(s)=D^sF_1(s)=\E(DX_1)^s$,
whence $X_2\eqd DX_1$ by \refC{C1}.

The final statement follows by considering $X_j\= e^{Y_j}$.
\end{proof}

\begin{remark}
  \label{Rcanonical}
The proofs of \refT{T+} and \refL{L+} yield an almost canonical way
of expressing $F(s)\in\cfo$ in the form \eqref{gammaF}. 
We start by making all $b_j,b_k'>0$ by \refT{T+}\ref{T+:}. We then
treat positive and negative $a_j$ and $a_k'$ separately; furthermore,
if these coefficients are not all commensurable (which they are in most
natural examples), we separate them into different equivalence classes of
commensurable coefficients. For each class we then rewrite the product
of the corresponding factors in the form \eqref{Fkill} for some real
$r$ (different for different classes, and chosen with $|r|$ as large as
possible). 
Note that  different factors in \eqref{Fkill}
have no common poles, so it is easy to locate all poles and zeros.
It only remains to take care of the rational part in
\eqref{Fkill}; in the examples we know, this is not a problem but we
have not studied this in general, and we do not know whether it is
possible to use this approach to define a unique canonical representation
\eqref{gammaF} for each $F$; we leave this as an open problem.
(Alternatively, it might be possible to define a canonical representation
including a rational factor.)

See Theorems \refand{T2}{TM} for examples of such 'canonical'
versions, but note that they not necessarily are the simplest by other
criteria. 
\end{remark}

 Let $N_+(x)\=\sum_{0<s\le x} \nu_F(s)$, $x>0$, and 
$N_-(x)\=\sum_{x\le  s<0} \nu_F(s)$, $x<0$. 
Thus $N_+(x)$ is the total number of poles minus the total number of
zeros (with multiplicities) in the interval $(0,x]$, and similarly for
$N_-(x)$ and the interval $[x,0)$ on the negative half-axis.
The following proposition can be interpreted as giving the density
of poles minus zeros on the positive or negative half-axis. 

\begin{proposition}\label{PN}
  Let $N_+(x)\=\sum_{0<s\le x} \nu_F(s)$, $x>0$, and 
$N_-(x)\=\sum_{x\le  s<0} \nu_F(s)$, $x<0$. Then 
$N_+(x)/x\to \half(\gam-\gam')$ as $x\topoo$
and
$N_-(x)/|x|\to \half(\gam+\gam')$ as $x\tomoo$.
\end{proposition}

\begin{proof}
  Use, for simplicity, a representation as in
  \refT{T+}\ref{T+:}. Then, using \eqref{nuf2}, 
the terms with $a_j>0$ and $a_k'>0$ give
no contributions to
  $\nu_F(s)$ and $N_+(x)$ for $s>0$ and $x>0$,
while each $a_j<0$ gives a contribution $|a_j|x+O(1)$ to $N_+(x)$
  (poles regularly spaced at distances $1/|a_j|$),
and similarly each $a_k'<0$ gives a contribution $-|a_k'|x+O(1)$ to
  $N_+(x)$.
Consequently, for $x>0$, using \refL{L1},
\begin{equation*}
%  \begin{split}
  N_+(x)=x\Bigpar{\sum_{j:a_j<0} |a_j| - \sum_{k:a_k'<0} |a_k'|} +O(1)
%\\&
=x\half(\gam-\gam')+O(1),	
%  \end{split}
\end{equation*}
and the result as $x\to\infty$ follows. The result as $x\tomoo$
follows similarly, or by replacing $F(s)$ by $F(-s)$.
\end{proof}

\section{Asymptotics of moments or \mgf}\label{Sasymp}

In this section we assume that $X>0$ and $Y=\log X$ are random
variables such that \eqref{gamma}--\eqref{gammachf} hold (for
$\rho_-<\Re s<\rho_+$ and $t\in\bbR$), \ie
\begin{equation}
  \label{gammaa}
\E X^s 
= \E e^{sY}
= F(s) 
=C D^s
\frac{\prodjj \Gamma(a_j s+b_j)}
{\prodkk \Gamma(\cx_k s+\dx_k) } ,
\end{equation}
and we write as above $D=e^d$. Recall the definitions \eqref{gam}--\eqref{CC1}.

We begin with asymptotics of $F$ along the imaginary axis and close to it.
%and, when possible, along the positive real axis.

\begin{theorem}
  \label{TIM}
%  \begin{thmenumerate}
%	\item
As $t\to\pm\infty$,
\begin{equation}\label{tim1}
  |\E e^{\ii tY}|=|F(\ii t)| \sim C_1 |t|^\gd e^{-\frac\pi2\gam|t|}.
\end{equation}
Moreover, for any fixed real $\gs$, and uniformly for $\gs$ in any
bounded set, 
\begin{equation}\label{tim2}
  |F(\gs+\ii t)|\sim e^{\gk\gs}|t|^{\gamma'\gs} |F(\ii t)|
\sim C_1 e^{\gk\gs}|t|^{\gd+\gamma'\gs} e^{-\frac\pi2\gam|t|}.
\end{equation}
%  \end{thmenumerate}
\end{theorem}

\begin{proof}
It is an easy, and well-known, consequence 
of Stirling's formula, see \eg{} \eqref{Astir'}, 
that for any
complex constant $c$ and all complex $z$ in a sector
$\abs{\arg  z}<\pi-\eps$ (where $\eps>0$) with $|z|$ large enough, 
for example $|z|\ge 2|c|/\sin\eps$,
\begin{equation}
  \label{stir2}
\log\gG(z+c)-\log\gG(z)=c\log z+\Oqwa z,
\end{equation}
uniformly for $c$ in any bounded set and such $|z|$.

If $a>0$ and $b\in\bbR$, we thus have for real $t\topoo$, taking
$z=\ii at$ in \eqref{stir2} and in Stirling's formula \eqref{Astir},
\begin{equation*}
  \begin{split}
	\log\gG(a\ii t+b)
&= 	\log\gG(\ii a t)+b\log(\ii a t) +\Oqw t 
\\
&= 	(\ii a t+b-\half)\log(\ii a t)-\ii a t+\lpi +\Oqw t 
\\
&= 	\bigpar{\ii a t+b-\half}\bigpar{\log(at)+\ii\pi/2}-\ii a t+\lpi +\Oqw t .
  \end{split}
\end{equation*}
%n this case, $\log(\ii at)=$, and 
Taking the real part, we find
\begin{equation*}
  \begin{split}
	\log|\gG(a\ii t+b)|
&=\Re (\log\gG(a\ii t+b))
\\&
= 	- \frac\pi2 a t+ (b-\half)\log(a t)+\lpi +\Oqw t .
  \end{split}
\end{equation*}
Consequently, for $a>0$,
\begin{equation}
%  \label{pan1}
|\gG(a\ii t+b)|
\sim \ppi\, a^{b-1/2} t^{b-1/2} e^{- \frac\pi2 a t}  ,
\qquad  t\topoo.
\end{equation}
For general real $a$ and $t$ we thus have (by $\gG(\bar z)=\overline{\gG(z)}$)
\begin{equation*}
  |\gG(a\ii t+b)| 
=\bigabs{\gG(|a|\ii|t|+b)}
\sim \ppi\, |a|^{b-1/2} |t|^{b-1/2} e^{- \frac\pi2 |a| |t|}  ,
\qquad  t\topmoo.
\end{equation*}
The result \eqref{tim1} follows by multiplying the various factors in
$F(\ii t)$ in 
\eqref{gammaa}, noting that $|D^{\ii t}|=1$.

For \eqref{tim2} we note that \eqref{stir2} implies
\begin{equation*}
  \begin{split}
  \log|\gG(a(\gs+\ii t)+b)| -   \log|\gG(a\ii t+b)|
&=\Re\bigpar{a\gs\log(a\ii t+b)+\Oqw t}
\\
&=\gs\bigpar{a\log|a|+a\log|t|+\Oqw t},
  \end{split}
\end{equation*}
and the result follows %similarly 
by multiplying the various factors in
$F(\gs+\ii t)/F(\ii t)$. 
(Alternatively, at least for fixed $\gs$, we may apply \eqref{tim1}
with $b_j$ replaced by $b_j+\gs a_j$,  $b'_k$ replaced by $b'_k+\gs a'_k$
and $C$ replaced by $C D^\gs=C e^{d\gs}$; note that the proof holds
for any function $F$ of this type, without assuming the existence of
random variables $X$ and $Y$.)
\end{proof}

\begin{proof}[Proof of \refP{P1}]
  By \eqref{tim1}, the values of $F(\ii t)$ determine $\gamma$, $\gd$
  and $C_1$. Further, choosing any fixed $\gs>0$ in \eqref{tim2}, we
  see that $F$ determines $\gamma'$ and $\gk$ too.
\end{proof}

\begin{corollary}
  \label{CIM}
%For any $X$ and $Y$ as above,
We have
$\gam\ge0$. Further, if\/ $\gam=0$, then $\gd\le0$.
\end{corollary}
\begin{proof}
  By letting $t\tooo$ in \eqref{tim1}, since $|\E e^{\ii tY}|\le1$.
\end{proof}

\begin{remark}
  If $\gam=0$ and $\gam'\neq0$, then \eqref{tim2} implies a better
  bound for $\gd$. However, we do not know any such example, and we
  leave it as an open problem whether there exists any $X$ with
  $\gam=0$ and $\gam'\neq0$.
\end{remark}

\begin{theorem}
  \label{TF}
If\/ $\gam>0$, then $X$ and $Y$ are absolutely continuous, with
continuous and infinitely differentiable density functions $f_X(x)$ on
$\ooo$ and $f_Y(y)$ on $\oooo$ given by
\begin{align}
  \label{fx}
f_X(x)&=\frac1{2\pi\ii}\intgs x^{-s-1}F(s)\dd s,\\
  \label{fy}
f_Y(y)&=\frac1{2\pi\ii}\intgs e^{-ys}F(s)\dd s,
\end{align}
for any $\gs\in(\rho_-,\rho_+)$.
\end{theorem}

\begin{proof}
  By \refT{TIM}, the \chf{} $\E e^{\ii tY}=F(\ii t)$ is integrable,
  and thus $Y$ has a continuous density $f_Y$ obtained by Fourier
  inversion:
  \begin{equation}\label{fy0}
f_Y(y)=
\frac1{2\pi}\intoooo e^{-\ii ty}F(\ii t)\dd t
=\frac1{2\pi\ii}\int_{-\ii\infty}^{\ii\infty} e^{-sy}F(s)\dd s.
  \end{equation}
Since $Y=\log X$, $X$ also is absolutely continuous, with the density function
  \begin{equation}\label{fx0}
	\begin{split}
f_X(x)=
\frac1x f_Y(\log x)
=\frac1{2\pi\ii}\int_{-\ii\infty}^{\ii\infty} x^{-s-1}F(s)\dd s.	  
	\end{split}
  \end{equation}
(Alternatively and equivalently, $F(s)$ is the Mellin transform of
  $f_X$, and this is the Mellin inversion formula.)

Since \refT{TIM} further implies that $|t|^NF(\ii t)$ is integrable
for every $N\ge0$, $f_Y$ and $f_X$ are infinitely differentiable and
we may differentiate \eqref{fx0} and \eqref{fy0} under the integral sign an
arbitrary number of times.

The integrands in \eqref{fx0} and \eqref{fy0} are analytic in $s$ for
$\rho_-<\Re s<\rho_+$, and thus the estimate \eqref{tim2} implies that
we can move the line of integration to any line $\Re s=\gs$ with
$\rho_-<\gs<\rho_+$. 
\end{proof}

\begin{remark}
  For $X$, we consider the density only for $x>0$, and `infinitely
  differentiable' here means on $\ooo$. Continuity and
  differentiability of $f_X$ at
  0 will be considered in \refT{T0}.
\end{remark}

\begin{remark}\label{R00}
  In the case $\gam=0$, the same argument shows that if $\gd<-1$, then
  $X$ and $Y$ have continuous density functions, which have at least
  $\ceil{|\gd|}-2$ continuous derivatives. However, 
  \refE{Ebeta}, where $\gd=-\gb$, shows that we in general do not
  have more derivatives. Similarly, Examples \refand{EU}{Ebeta} show
  that we do not necessarily have continuous density functions for
  $\gam=0$ and $-1\le\gd\le0$. \refE{Ecounter} gives an example with
  $\gam=\gd=0$ where the distribution is mixed with a point mass
  besides the absolutely continuous part.

Note that, by \eqref{tim1}, $\gam=\gd=0$ if and only if $|\E X^{\ii
  t}|=|\E e^{\ii t Y}|$ has 
a non-zero limit as $t\to\pm\infty$; by the
Riemann--Lebesgue lemma, this implies that $Y$ and $X$ do not have
  absolutely continuous distributions.
\end{remark}

We next consider asymptotics of $F$ along the real axis, when
possible. 

\begin{theorem}
  \label{TRE}
  \begin{thmenumerate}
\item
If\/ $\rho_+=\infty$, then for real $s\to+\infty$,
\begin{equation}
  \E X^s = \E e^{sY} = F(s)
\sim C_1 s^\gd e^{\gam s \log s + (\gk-\gamma)s}.
\end{equation}
\item
If\/ $\rho_-=\infty$, then for real $s\to-\infty$,
\begin{equation}
  \E X^s = \E e^{sY} = F(s)
\sim C_1 |s|^\gd e^{\gam |s| \log |s| + (\gk+\gamma)s}.
\end{equation}
  \end{thmenumerate}
\end{theorem}

\begin{proof}
  \pfitem i
By  \refT{T+} and \refP{P1}, we may assume that all $a_j,a'_k>0$.
We then argue as for \refT{TIM}. If $a>0$ and $b\in\bbR$, then for
real $s\topoo$,
\begin{equation}\label{sofie}
  \begin{split}
\log\gG(as&+b) =\log\gG(as)+b\log(as)+\Oqw s
\\	
&= 	(a s+b-\half)\log(as)- as+\lpi +\Oqw s
\\
&= as\log s +(a\log a-a)s+(b-\half)\log s
\\&\hskip8em{}+(b-\half)\log a+\lpi +\Oqw s,
%\raisetag{\baselineskip}
  \end{split}
\end{equation}
and the result follows again by multiplying the factors.

\pfitem{ii}
This follows from (i) by replacing $Y$ by $-Y$, see \refR{R+-}:
\begin{equation*}
  \E e^{sY} = 
\E e^{|s|(-Y)} 
\sim C_1 |s|^\gd e^{\gam |s| \log |s| + (-\gk-\gamma)|s|}.
\qedhere
\end{equation*}
\end{proof}

If $\rho_+<\infty$, then $F$ has poles, and possibly zeros, on the
positive real axis. Typically, there is an infinite number of such
poles
(but see \refE{EU} for a counter example), and then we cannot
consider asymptotics for all $s\topoo$. However, we can restrict
$s$ to a subset of $\bbR$ and obtain asymptotic results similar to
\refT{TRE} in this case too.

\begin{lemma}
  \label{LE}
Given real $a_j,b_j,a'_k,b_k'$ 
for $1\le j\le J$ and $1\le k\le K$, with $a_j,a_k'\neq0$,
there exists a closed set $E\subset\bbR$ 
and a constant $\xi>0$ 
such that $E\cap I$ has
measure greater than $1/2$ for every interval $I$ of length $1$, and 
$|\sin(\pi(a_js+b_j))|\ge\xi$ and $|\sin(\pi(a_k's+b_k'))|\ge\xi$
for every $j$ and $k$ and all $s\in E$.
\end{lemma}

\begin{proof}
  Let $N$ be the set of all (real) $s$ such that 
$a_js+b_j\in\bbZ$ for  some $j$  
or $a_k's+b_k'\in\bbZ$ for  some $k$. There exists a constant $M$
such that no interval of length 1 contains more than $M$ points of
$N$. (For example, $M=J+K+\sum_j|a_j|\qw+\sum_k|a_k'|\qw$.)
It follows that $E\=\set{x:|x-s|\ge 1/(2M+3)\text{ for all $s\in N$}}$
satisfies the properties, for some $\xi>0$.
%Then $|I\cap E|>1/2$ for every interval $I$ of length |I|=1, and
%there exists a   
\end{proof}

In the sequel we let $E$ denote this set, defined for a given
representation \eqref{gammaa} of $F(s)$.
By considering only $s\in E$, we can  extend \refT{TRE} to arbitrary $F$.

\begin{theorem}
  \label{TREX}
For real $s\to\pm\infty$ with $s\in E$,
\begin{equation}
 | F(s)|
= |s|^\gd e^{\gam' s \log |s| + (\gk-\gamma')s+O(1)}.
\end{equation}
\end{theorem}

\begin{proof}
If $a<0$ and $b\in\bbR$, then for real $s\to+\infty$, by \eqref{Asin}
and \eqref{sofie},
\begin{equation}%\label{emma}
  \begin{split}
\log|&\gG(as+b)| =
-\log|\gG(|a|s-b+1)|+\log\pi-\log|\sin(\pi(as+b))|
\\
&=-|a|s\log s -(|a|\log |a|-|a|)s-(\half-b)\log s-(\half-b)\log |a|
\\&\hskip10em{}+\log\sqrt{\pi/2} -\log|\sin(\pi(as+b))|+\Oqw s,
\\
&= as\log s +(a\log |a|-a)s+(b-\half)\log s
-\log|\sin(\pi(as+b))|
+O(1).
  \end{split}
\end{equation}
If $(a,b)$ is some $(a_j,b_j)$  or $(a_k',b_k')$ with $a<0$,
we thus have by \refL{LE}, for $s\in E$ with $s\topoo$,
\begin{equation}\label{julie}
  \begin{split}
\log|&\gG(as+b)| 
= as\log s +(a\log |a|-a)s+(b-\half)\log s
+O(1).
  \end{split}
\end{equation}
By \eqref{sofie}, \eqref{julie} holds also for $a>0$ (and all
$s\topoo$).

Further, replacing $s$ by $-s$ and $a$ by $-a$ in \eqref{julie}, we see that
if $(a,b)$ is some $(a_j,b_j)$  or $(a_k',b_k')$,
then for $s\in E$ with $s\tomoo$,
\begin{equation}\label{jw}
  \begin{split}
\log|&\gG(as+b)| 
= as\log |s| +(a\log |a|-a)s+(b-\half)\log |s|
+O(1).
  \end{split}
\end{equation}
Thus \eqref{jw} holds for all such $(a,b)$ and $s\to\pm\infty$ with
$s\in E$, and the result follows from \eqref{gammaa}.
\end{proof}

Note that if $\rho_+=\infty$, then $\gam'=\gam$, while
if $\rho_-=\infty$, then $\gam'=-\gam$ 
by \eqref{gam}--\eqref{gam'} together with \refT{T+} and \refP{P1};
hence the exponents in Theorems \ref{TRE} and \ref{TREX} agree (as they must).

For complex arguments, we will use the following estimate.

\begin{lemma}
  \label{LC}
Let $\Psi(\gs,t)\=\int_0^t\arctan(u/\gs)\dd u$ for $t\ge0$.
Then, for $\gs>0$ with $\gs\in E$ and all real $t$,
\begin{equation*}
  \frac{|F(\gs+\ii t)|}{|F(\gs)|} 
=\exp\Bigpar{-\frac{\pi}2(\gam-\gam')|t|-\gam'\Psi(\gs,|t|)+O(1+|t|\gs\qw)}.
\end{equation*}
\end{lemma}

\begin{proof}
  We may assume $t>0$.
Consider first a factor $\gG(as+b)$ with $a>0$ and $s=\gs+\ii t$,
$\gs>0$. (We may assume that $\gs$ is large so that $a\gs+b>0$, \eg\
by using \eqref{tim2} for small $\gs$.)
By \eqref{Astir'},
\begin{equation}
  \begin{split}
\frac{\dd}{\dd t}\log|\gG(a(\gs+\ii t)+b)|	
&=\Re \frac{\dd}{\dd t}\log\gG(a(\gs+\ii t)+b)
\\&
=\Re \bigpar{\ii a\log(a(\gs+\ii t)+b)}+\Oqwa \gs
\\&
=\Re \bigpar{\ii a\log(a(\gs+\ii t))}+\Oqwa \gs
\\&
=-a\Im \bigpar{\log(\gs+\ii t)}+\Oqwa \gs
\\&
=-a \arctan(t/\gs)+\Oqwa \gs.
  \end{split}
\end{equation}
Consequently, integrating from 0 to $t$,
\begin{equation}\label{emma}
\log|\gG(a(\gs+\ii t)+b)|-\log|\gG(a\gs+b)|	
=-a\Psi(\gs,t)+O(t\gs\qw).
\end{equation}

If $a<0$, we argue as in the proof of \refT{TREX} and have by
\eqref{Asin}
\begin{multline*}
  \log|\gG(a(\gs+\ii t)+b)|
\\
=  -\log|\gG(|a|(\gs+\ii t)+1-b)|
-\log|\sin(\pi(a(\gs+\ii t)+b))|+\log\pi.
\end{multline*}
If further $(a,b)=(a_j,b_j)$ or $(a_k',b_k')$ for some $j$ or $k$, and
$\gs\in E$, then
$\log|\sin(\pi(a(\gs+\ii t)+b))|=\pi|a|t+O(1)$, and 
it follows, using \eqref{emma} with $|a|$ instead of $a$, that
\begin{equation*}%\label{emma-}
  \begin{split}
\log|\gG(a(\gs+\ii t)+b)|-\log|\gG(a\gs+b)|	
&=|a|\Psi(\gs,t)-\pi|a|t+O(1+t\gs\qw)
\\
&=-a\Psi(\gs,t)+\pi at+O(1+t\gs\qw).	
  \end{split}
\end{equation*}
The result follows by multiplying the factors in $F$, using \refL{L1}.
\end{proof}

\section{Asymptotics of density function}\label{Sdensity}
We continue to assume that $X$ and $Y=\log X$ are random variables
such that \eqref{gamma}--\eqref{gammachf} hold; as above we write 
$\E X^s=e^{sY}=F(s)$. 
We assume $\gam>0$, so that density functions of $X$ and $Y$ exist by
\refT{TF}, and consider asymptotics of the density function $f_X(x)$ as
$x\to0$ or $x\tooo$, or equivalently of $f_Y(y)$ as $y\to-\infty$ or $y\tooo$.
By symmetry it suffices to consider one side, and we concentrate on
$x\tooo$, but for convenience in applications we write most results
for both sides and for both $X$ and $Y$.

We consider first $x\to\infty$ ($y\to\infty$) and
begin with the case $\rho_+=\infty$, when $X$ has moments of all
(positive) orders and $f_X$ decreases rapidly (as we will see in
detail soon). We use the saddle point method, see \eg{}
\citetq{Chapter VIII}{FS},  
in a standard way.

\begin{theorem}
  \label{TAinfty}
Suppose that $\rho_+=\infty$ and $\gam>0$. Then
\begin{align*}
  f_X(x)&\sim C_2 x^{c_1-1}e^{-c_2 x^{1/\gam}},
&& x\tooo,
\\
f_Y(y)&\sim \frac{C_1}{\sqrt{2\pi\gam}} e^{c_1(y-\gk)-\gam e^{(y-\gk)/\gam}},
&& y\tooo,
%\end{align*}
\intertext{where}
%\begin{align*}
  c_1&\=(\gd+1/2)/\gam,\\
  c_2&\=\gam e^{-\gk/\gam},\\
  C_2&\=\frac{C_1}{\sqrt{2\pi\gam}}e^{-c_1\gk}.
\end{align*}
\end{theorem}

\begin{proof}
  By \refT{T+} and \refP{P1}, we may assume that all
  $a_j,a_k',b_j,b_k'>0$. We will use \eqref{fx}, which now is valid
  for all $x>0$ and $\gs>0$.

By \eqref{Astir'} and \eqref{Astir''}, for $a,b>0$ and $\Re s>0$,
\begin{align*}
%  \frac{\dd}{\dd s}
\bigpar{\log\gG(as+b)}'&
=a\log (as+b)+O\bigpar{|s|\qw}
=a\log (as)+O\bigpar{|s|\qw},
\\
%  \frac{\dd^2}{\dd s^2}
\bigpar{\log\gG(as+b)}''&=\frac as+O\bigpar{|s|\qww}.
\end{align*}
Consequently, writing 
\begin{equation*}
f(s)\=\log F(s)=\log C+ds+  \sumjj \log\Gamma(a_j s+b_j)
-\sumkk \log\Gamma(\cx_k s+\dx_k) ,
\end{equation*}
we have for $\Re s>0$,
\begin{align*}
f'(s)&
=d+\sumjj (a_j\log a_j+a_j\log s)
-\sumkk (a_k'\log a_k'+a_k'\log s)
+O\bigpar{|s|\qw}
\\&
=\gk+\gam\log s+O\bigpar{|s|\qw},
\\
f''(s)&=  \frac{\gam}s+O\bigpar{|s|\qww}.
\end{align*}
Fix $x>0$ and let $G(s)\=x^{-s-1}F(s)$ and $g(s)\=\log G(s)=-(s-1)\log x+f(s)$.
Then
\begin{equation}
  g'(s)=f'(s)-\log x=
\gam\log s+\gk-\log x+O\bigpar{|s|\qw}.
\end{equation}
We choose (for $x$ large)
$\gs=e^{(\log x-\gk)/\gam}$, so $\gam\log\gs=\log x-\gk$ and
$g'(\gs)=O(\gs\qw)$; thus $\gs$ is an approximate saddle point of
$G(s)$. 
Note that $\gs\tooo$ as $x\tooo$, so $\gs\qw\to0$.
Since $g''(s)=f''(s)$, we further have, as \xtoo,
\begin{equation}
  g''(\gs)=\gam\gs\qw+O\bigpar{\gs\qww}\sim\gam\gs\qw.
\end{equation}
Further, on the line $\Re s=\gs$, 
\begin{equation}
  g''(\gs+\ii t)=f''(\gs+\ii t)
=\frac{\gam}{\gs+\ii t}+O\bigpar{\gs\qww}
=\frac{\gam}{\gs}+O\bigpar{(1+\abs t)\gs\qww}.
\end{equation}
Consequently, Taylor's formula yields
\begin{equation}
  g(\gs+\ii t)=g(\gs)+O\bigpar{|t|\gs\qw}
-\frac{\gam}{2\gs}t^2+O\bigpar{(|t|^2+|t|^3)\gs\qww}
\end{equation}
and, uniformly for $|t|\le\gs^{0.6}$,
\begin{equation}\label{ika}
  G(\gs+\ii t)=G(\gs)e^{-\gam t^2/2\gs+o(1)}.
\end{equation}
For larger $|t|$, we have a rapid decay, for example by \refL{LC}
which yields, for large $\gs$ and $|t|\ge\gs^{0.6}$, 
recalling that now $\gam'=\gam$,
\begin{equation}\label{magn}
  \begin{split}
  \frac{|G(\gs+\ii t)|}{|G(\gs)|} 
&=
  \frac{|F(\gs+\ii t)|}{|F(\gs)|} 
=\exp\Bigpar{-\gam\Psi(\gs,|t|)+O(1+|t|\gs\qw)}
\\&
\le \exp\Bigpar{-c\min(|t|,|t|^2/ \gs)}.
  \end{split}
\end{equation}
for some $c>0$.
It follows from \eqref{ika} and \eqref{magn} that \eqref{fx} yields,
using \refT{TRE} and the choice of %$\gs$,
$\gs=e^{(\log x-\gk)/\gam}=e^{-\gk/\gam}x^{1/\gam}$, 
%$\gam\log\gs=\log x-\gk$ 
\begin{equation*}
  \begin{split}
f_X(x)&=\frac1{2\pi}\intoooo G(\gs+\ii t)\dd t
\sim
\frac{G(\gs)}{2\pi}\intoooo e^{-\gam t^2/2\gs}\dd t
=
\frac{G(\gs)\sqrt{\gs}}{\sqrt{2\pi\gam}}
\\&
=
\frac{x^{-\gs-1}F(\gs){\gs}\qq}{\sqrt{2\pi\gam}}
\sim
\frac{C_1}{\sqrt{2\pi\gam}}\gs^{\gd+1/2}x\qw
 e^{\gam\gs\log\gs+(\gk-\gam)\gs-\gs\log x}
\\&
=
\frac{C_1}{\sqrt{2\pi\gam}}e^{(\gd+1/2)(\log x-\gk)/\gam}x\qw
 e^{-\gam\gs}
=C_2x^{c_1-1}e^{-c_2x^{1/\gam}}.
  \end{split}
\end{equation*}

The result for $f_Y$ follows similarly from \eqref{fy}, or simpler by
$f_Y(y)=e^yf_X(e^y)$. 
\end{proof}

\begin{remark}
  \label{Rinfty'}
The derivative $f_X'(x)$ and higher derivatives $f_X\xn(x)$ can be obtained by
  repeated differentiation of \eqref{fx} under the integral sign,
  which multiplies the integrand by a factor $(-s-1)\dotsm(-s-n)x^{-n}$.
The argument above, including the estimates \eqref{ika} and \eqref{magn},
  applies to this integral as well and shows that, for any $n\ge0$,
  \begin{equation*}
f_X\xn(x)
\sim\gs^nx^{-n}G(\gs)\sqrt{\gs}/\sqrt{2\pi\gam}	
\sim\gs^nx^{-n}f_X(x)
=(c_2/\gam)^nx^{n(1/\gam-1)}f_X(x).
  \end{equation*}
In particular, every derivative of $f_X$ tends to 0 rapidly 
(faster than any power of $x$)
as \xtoo.
\end{remark}

\begin{remark}\label{Rhighersaddle}
  The saddle-point method yields also
more precise asymptotics  including higher-order terms
by refining the estimates around $s=\gs$ in the proof above, 
see \eg{} \citetq{Section VIII.3}{FS}; 
we leave the details to the reader. This yields an asymptotic
expansion in powers of $\gs\qw$, with $\gs$ as in the proof above,
\ie, in powers of $x^{-1/\gam}$. See \refR{RF20} for an example of
such an expansion (there obtained from a known result rather than by
performing the calculations).
\end{remark}

We continue with the case $\rho_+<\infty$, when $F(s)$ has a pole at
$\rho_+$ of order $\nu_F(\rho_+)\ge1$. 
(Recall the notion $\nu_F$ from \eqref{nuf}.)
We denote the coefficients of the singular part of the Laurent
expansion of $F$ at a point $s_0$ by $c_\ell(s_0)$:
\begin{equation}\label{laurent}
  F(s)=\sum_{\ell=1}^{\absnuf{s_0}}c_\ell(s_0)(s-s_0)^{-\ell}+O(1)
\qquad\text{as $s\to s_0$}.
\end{equation}
%We let $c_\ell(s_0)=0$ if $\ell>\absnuf{s_0}$, and in particular for
In particular, $c_{1}(s_0)$ is the residue $\res{s_0} F$.

We have the following standard result by Mellin inversion, see 
\cite{FGD}. 

\begin{theorem}
  \label{TAfinite}
Suppose that $\rho_+<\infty$ and $\gam>0$.
\begin{romenumerate}
  \item
As \xtoo, for some $\eta>0$,
\begin{equation*}
  f_X(x)=x^{-\rho_+-1}\sum_{\ell=0}^{\absnuf{\rho_+}-1}
\frac{(-1)^{\ell+1} c_{\ell+1}(\rho_+)}{\ell!} \log^\ell x
+O\bigpar{x^{-\rho_+-1-\eta}}
\end{equation*}
In particular, with $\nu\=\absnuf{\rho_+}\ge1 $,
\begin{equation*}
  f_X(x)\sim 
\frac{(-1)^{\nu} c_{\nu}(\rho_+) }{(\nu-1)!}
x^{-\rho_+-1}\log^{\nu-1} x
.
\end{equation*}
If\/ $\rho_+$ is a simple pole of $F$, \ie{} $\nu=1$, this can be written
\begin{equation*}
  f_X(x)\sim 
-\res{\rho_+}(F)\, x^{-\rho_+-1} .
\end{equation*}
%(where the residue is negative).

\item
More precisely,
there is an asymptotic expansion, for any fixed $\gs>0$,
\begin{equation*}
  f_X(x)
=\sum_{0<\rho\le\gs}\sum_{\ell=0}^{\nu_F\xpar{\rho}-1}
\frac{(-1)^{\ell+1} c_{\ell+1}(\rho)}{\ell!} x^{-\rho-1}\log^\ell x
+O\bigpar{x^{-\gs-1}},
\end{equation*}
summing over all poles $\rho$ of $F$ in $(0,\gs]$. (The inner sum
  vanishes unless $\rho$ is a pole, so formally we may sum over all $\rho$.)
\end{romenumerate}
Corresponding asymptotics for $f_Y(y)=e^yf_X(e^y)$ are obtained by
replacing each $x^{-r-1}$ by $e^{-r y}$ and $\log^\ell x$ by $y^\ell$.
\end{theorem}

\begin{proof}
  As said above, this is a standard result, and we refer to \cite{FGD}
  for details, but for completeness and
  later use we give the simple proof. 

It suffices to prove (ii), since (i) follows by taking $\gs=\rho_++\eta$.
We may assume that $\gs$ is not a pole of $F$ (otherwise we increase
$\gs$ a little). We start with \eqref{fx}, where we integrate over a
line with $\Re s\in(\rho_-,\rho_+)$. We may, using \refT{TIM}, shift
the line to $\Re s=\gs>\rho_+$ too, but then we have to subtract the
residues of the traversed poles. Thus
\begin{equation}\label{52}
f_X(x)=\frac1{2\pi\ii}\intgs x^{-s-1}F(s)\dd s
-\sum_{0<\rho\le\gs}\res{s=\rho}(x^ {-s-1}F(s) ),
\end{equation}
and the result follows by computing the residues, using
\eqref{laurent} and $x^{-s-1}=x^{-\rho-1}\sum_{\ell=0}^\infty(-\log
x)^\ell(s-\rho)^\ell/\ell!$,
and noting that, by \refT{TIM} again, 
\begin{equation*}
 \lrabs{\intgs x^{-s-1}F(s)\dd s} 
\le 
 x^{-\gs-1}\intgs |F(s)|\dd s
= O\bigpar{ x^{-\gs-1}}.
\qedhere
\end{equation*}
\end{proof}

In \refT{TAfinite}(ii) we have an asymptotic expansion, valid for
fixed $\gs$ as $x\to\infty$. It is natural to ask whether this
asymptotic expansion actually yields a series representation for
$f_X(x)$, \ie, whether we can let $\gs\to\infty$ for fixed $x$ 
(with the error term tending to 0) so that $f_X(x)$ is represented as a
convergent series.
This is possible sometimes, but not always. In fact, the following
theorem shows that this is possible exactly when $\gam'<0$, at least
provided that there is an infinite number of poles $\rho>0$ and that
these are simple.

\begin{theorem}
  \label{TAsum}
Suppose that  $\gam>0$.
\begin{thmenumerate}
  \item\label{TAsum-}
If\/ $\gam'<0$,
then, for all $x>0$,
\begin{equation}\label{tasum}
  f_X(x)
=\sum_{\rho>0}\sum_{\ell=0}^{\nu_F\xpar{\rho}-1}
\frac{(-1)^{\ell+1} c_{\ell+1}(\rho)}{\ell!} x^{-\rho-1}\log^\ell x,
\end{equation}
summing over all poles $\rho>0$ of $F$. 
In particular, if $F$ has only simple poles,
\begin{equation}\label{sumres}
  f_X(x)
=\sum_{\rho>0}
-\res{\rho}(F) x^{-\rho-1}.
\end{equation}
\item\label{TAsum+}
If\/ $\gam'>0$ and there is an infinite number of poles $\rho>0$ of
$F$, all simple, then the sum \eqref{sumres} diverges for all $x>0$.
\item\label{TAsum0}
If\/ $\gam'=0$, then \eqref{tasum} holds for $x>e^\gk$; hence
 \eqref{sumres} holds for $x>e^\gk$ provided all poles are simple.
However, at least
provided that
there is an infinite number of poles $\rho>0$ of
$F$ and all such poles are simple,
 the sum \eqref{sumres}
diverges for $0<x<e^\gk$.
\end{thmenumerate}

Corresponding results for $f_Y(y)$ are obtained by
replacing $x^{-\rho-1}$ by $e^{-\rho y}$
and $\log^\ell x$ by $y^\ell$. The cut-offs in \ref{TAsum0} become $y>\gk$
and $y<\gk$.
\end{theorem}

\begin{proof}
\pfitemx{\ref{TAsum-} and \ref{TAsum0} (convergence)}
  We use again \eqref{52}, and have to show that the integral tends to
  0 as $\gs\to\infty$ for every fixed $x>0$.
We use \refL{LC}, and note that 
$0\le\Psi(\gs,t)\le \frac\pi2 t$ for $t\ge0$, and thus, because $\gam'\le0$,
for $\gs>0$ with $\gs\in E$,
\begin{equation*}
  \begin{split}
  \frac{|F(\gs+\ii t)|}{|F(\gs)|} 
&\le\exp\Bigpar{-\frac{\pi}2(\gam-\gam')|t|-\gam'\frac\pi2|t|+O(1+|t|\gs\qw)}
\\&
=\exp\Bigpar{-\frac{\pi}2\gam|t|+O(1+|t|\gs\qw)}.	
  \end{split}
\end{equation*}
If $\gs\in E$ is large enough we thus have 
$|F(\gs+\ii t)| = O\bigpar{e^{-\gam|t|}|F(\gs)|}$ for all real $t$,
and hence, because $\gam>0$,
\begin{equation*}
\intgs\bigabs{x^{-s-1}F(s)}\,|\dd s| 
= O\bigpar{x^{-\gs-1}|F(\gs)|}\intoooo e^{-\gam|t|}\dd t
= O\bigpar{x^{-\gs-1}|F(\gs)|}.
\end{equation*}
If $\gam'<0$, then this is by \refT{TREX} $o(1)$ as $\gs\to\infty$ 
for any fixed $x>0$, which shows \ref{TAsum-}.

If $\gam'=0$, then \refT{TREX} yields, for a fixed $x>0$, 
$x^{-\gs}F(\gs)=O(\gs^\gd e^{(\gk-\log x)\gs})$, which is $o(1)$ for
$x>e^\gk$, showing the positive part of (iii).

\pfitemx{\ref{TAsum+} and \ref{TAsum0} (divergence)}
  Let $\rho>0$ be a pole of $F$ that is not too close to a zero or
  another pole, 
meaning that the distance to every zero or other pole is at least some
  small constant $\xi>0$. (This is true for all poles if all $a_j,a_k'$
  are commensurable and $\xi$ is small enough; in general it is true
  for a large fraction of the poles, and certainly an infinite number
  of them.) A simple modification of the proof of of \refT{TREX} then
  yields the same estimate as there for the residue at $\rho$:
\begin{equation*}
|\res\rho(F)|
= \rho^\gd e^{\gam' \rho \log \rho + (\gk-\gamma')\rho+O(1)}.
\end{equation*}
and thus, for every fixed $x$,
\begin{equation*}
|x^{-\rho-1}\res\rho(F)|
= \rho^\gd e^{\gam' \rho \log \rho + (\gk-\gamma'-\log x)\rho+O(1)}.
\end{equation*}
Letting $\rho\to\infty$, we see that the terms of \eqref{sumres}
are unbounded if
$\gam'>0$ or $\gam'=0$ and $\gk>\log x$; hence the sum diverges.
\end{proof}

\begin{remark}\label{RTAsumdiv}
To show divergence in \ref{TAsum+} and \ref{TAsum0}, we assumed 
for simplicity that $F$ has only simple poles on the positive axis; we
conjecture that, more generally, \eqref{tasum} diverges also without
this restriction. 

To show divergence we also assumed that $F$ has an infinite number of
positive poles; this is, on the contrary, obviously necessary for
divergence, since otherwise the sums \eqref{tasum} and \eqref{sumres}
are finite. However, if $F$ has only a finite number of positive poles,
then the sum in \eqref{tasum} or \eqref{sumres} is not integrable, since it is 
$\sim c x^{-\rho-1}\log^\ell x$ as $x\to0$, where $\rho>0$ is the
largest pole of $F$ and $c\neq0$, $\ell=\nu_F(\rho)-1$; hence the sum
cannot equal $f_X(x)$ for all $x>0$. \refE{ETAsum} yields an example
where the sum does not equal $f_X(x)$ for any $x>0$ (although the
difference tends to 0 rapidly as $\xtoo$  by \refT{TAfinite}).

In this connection, note that if $\gam>\gam'$, then 
there is an infinite number of poles in $(0,\infty)$
by \refP{PN}.
\end{remark}

We now consider $x\to0$ and $y\to-\infty$. We obtain the following by
the same methods as above 
(now moving the line of integration towards $-\infty$), or more simply
by applying the results above to $X\qw$ and $-Y$; this replaces $F(s)$ by
$F(-s)$ and the Laurent coefficients $c_\ell(s_0)$ by $(-1)^\ell c_\ell(-s_0)$.

\begin{theorem}
  \label{TAinfty0}
Suppose that $\rho_-=-\infty$ and $\gam>0$. Then
\begin{align*}
  f_X(x)&\sim C_3 x^{-c_1-1}e^{-c_3 x^{-1/\gam}},
&& x\to0,
\\
f_Y(y)&\sim \frac{C_1}{\sqrt{2\pi\gam}} e^{-c_1(y-\gk)-\gam e^{-(y-\gk)/\gam}},
&& y\tomoo,
%\end{align*}
\intertext{where}
%\begin{align*}
  c_1&\=(\gd+1/2)/\gam,\\
  c_3&\=\gam e^{\gk/\gam},\\
  C_3&\=\frac{C_1}{\sqrt{2\pi\gam}}e^{c_1\gk}.
\end{align*}
\end{theorem}

\begin{theorem}
  \label{TAfinite0}
Suppose that $\rho_->-\infty$ and $\gam>0$.
\begin{romenumerate}
  \item
As $x\downto0$, for some $\eta>0$,
\begin{equation*}
  f_X(x)=x^{|\rho_-|-1}\sum_{\ell=0}^{\absnuf{\rho_-}-1}
\frac{ c_{\ell+1}(\rho_-)}{\ell!} \log^\ell (1/x)
+O\bigpar{x^{|\rho_-|-1+\eta}}
\end{equation*}
In particular, with $\nu\=\absnuf{\rho_-}\ge1 $,
\begin{equation*}
  f_X(x)\sim 
\frac{ c_{\nu}(\rho_-) }{(\nu-1)!}
x^{|\rho_-|-1}\log^{\nu-1}(1/ x)
.
\end{equation*}
If\/ $\rho_-$ is a simple pole of $F$, \ie{} $\nu=1$, this can be written
\begin{equation*}
  f_X(x)\sim 
\res{\rho_-}(F)\, x^{|\rho_-|-1} .
\end{equation*}

\item
More precisely,
there is an asymptotic expansion, for any fixed $\gs>0$,
\begin{equation*}
  f_X(x)
=\sum_{0>\rho\ge-\gs}\sum_{\ell=0}^{\nu_F\xpar{\rho}-1}
\frac{ c_{\ell+1}(\rho)}{\ell!} x^{|\rho|-1}\log^\ell (1/x)
+O\bigpar{x^{\gs-1}},
\end{equation*}
summing over all poles $\rho$ of $F$ in $[-\gs,0)$. 
%(The inner sum
%  vanishes unless $\rho$ is a pole, so formally we may sum over all $\rho$.)
\end{romenumerate}
Corresponding asymptotics for $f_Y(y)=e^yf_X(e^y)$ are obtained by
replacing each $x^{r-1}$ by $e^{r y}$ and $\log^\ell(1/ x)$ by $(-y)^\ell$.
\end{theorem}

\begin{theorem}
  \label{XTAsum}
Suppose that  $\gam>0$.
\begin{thmenumerate}
  \item\label{XTAsum+}
If\/ $\gam'>0$,
then, for all $x>0$,
\begin{equation}\label{xtasum}
  f_X(x)
=\sum_{\rho<0}\sum_{\ell=0}^{\nu_F\xpar{\rho}-1}
\frac{c_{\ell+1}(\rho)}{\ell!} x^{|\rho|-1}\log^\ell(1/ x),
\end{equation}
summing over all poles $\rho<0$ of $F$. 
In particular, if $F$ has only simple poles,
\begin{equation}\label{xsumres}
  f_X(x)
=\sum_{\rho<0}
\res{\rho}(F) x^{|\rho|-1}.
\end{equation}
\item\label{XTAsum-}
If\/ $\gam'<0$ and there is an infinite number of poles $\rho<0$ of
$F$, all simple, then the sum \eqref{sumres} diverges for all $x>0$.
\item\label{XTAsum0}
If\/ $\gam'=0$, then \eqref{xtasum} holds for $0<x<e^\gk$; hence
 \eqref{xsumres} holds for $0<x<e^\gk$ provided all poles are simple.
However, at least
 provided that
there is an infinite number of poles $\rho<0$ of
$F$ and all such poles are simple,
 the sum \eqref{xsumres}
diverges for $x>e^\gk$.
\end{thmenumerate}

Corresponding results for $f_Y(y)$ are obtained by
replacing $x^{|\rho|-1}$ by $e^{|\rho|y}$
and $\log^\ell (1/x)$ by $(-y)^\ell$. 
The cut-offs in \ref{XTAsum0} become $y<\gk$
and $y>\gk$.
\end{theorem}

Theorems \refand{TAsum}{XTAsum} say that (at least if $\gam>0$),
$f_X(x)$ has a series expansion in positive (but not necessarily integer)
powers of $x$ if $\gam'>0$,
and a series expansion in negative (but not necessarily integer)
powers of $x$ if $\gam'<0$, in both cases allowing for terms with
logarithmic factors too; if $\gam'=0$ one expansion holds for $0<x<e^\gk$
and the other for $x>e^\gk$.

\begin{remark}
  Suppose that all $a_j,a_k'$ are commensurable; then $F(s)$ may as in
  \refS{Spoles} (see the proof of \refL{L+}) be rewritten with all
  $a_j,a_k'=\pm r$, for some real $r>0$.
The poles of $F$ in $(-\infty,0)$ then form one or several arithmetic
  series \set{s_j-n/r} with gap $1/r$, possibly apart from a finite
  number of other poles. If further all poles are simple, then the
  residue at such a pole $s_l-n/r$ is of the form 
$(C/r)(-D)^n (n!)\qw\prod_{j\neq l} \gG(n+c_j)/\prod_k(\gG(n+c_k')$,
and the contribution to \eqref{xsumres}
from this series of poles is a (generalized)
hypergeometric series with argument $-D x^{1/r}$, times a constant and
  a power of $x$. Consequently, if further $\gam,\gam'>0$, then
the density function may be expressed using one or several
  hypergeometric functions. 
Typical examples are given in Theorems \refand{Tarea}{TMdensity}.
\end{remark}

As a corollary, we get results on continuity and differentiability at 0.

\begin{theorem}\label{T0}
  Suppose that $\gam>0$. 
  \begin{thmenumerate}
\item
The density $f_X$ is continuous at $0$, and thus everywhere on $\bbR$,
if and only if $\rho_-<-1$.	
\item
The density $f_X$ has a finite jump at $0$
if and only if $\rho_-=-1$ and this is a simple pole of $F$. 
In this case $f_X(0+)=\res{\rho_-}(F)$.
\item
The density $f_X$ is infinitely differentiable on $\bbR$
if and only if $\rho_-=-\infty$.	
  \end{thmenumerate}
\end{theorem}

\begin{proof}
Note that $f_X$ is infinitely differentiable on $(0,\infty)$ by
\refT{TF}, as well as, trivially, on $(-\infty,0)$ where it vanishes.

 Parts (i) and (ii) follow immediately from Theorems
  \refand{TAinfty0}{TAfinite0}.

If $f_X$ is infinitely differentiable at 0, then every derivative
$f_X\xn(0)=0$ because $f_X$ vanishes on $(-\infty,0)$. Hence a Taylor
expansion shows that $f_X(x)=O(x^N)$ as $x\to0$ for every integer $N$.
If $\rho_-$ were finite, this would contradict \refT{TAfinite0}; hence
$\rho_-=-\infty$. 

Conversely, if $\rho_-=-\infty$, then \refT{TAinfty0} shows that
$f_X(x)$ tends to 0 rapidly as $x\downto0$. Moreover, by
\refR{Rinfty'} and the usual change of variables $x\mapsto1/x$,
the same holds for each derivative $f_X\xn(x)$. It follows, by
induction, that each derivative $f_X\xn(x)$ exists also at $x=0$
with $f_X\xn(0)=0$.
Hence $f_X$ is infinitely differentiable.
\end{proof}
\begin{remark}
More generally, $f_X$ has $n$ continuous derivatives (at 0) if and only
if $\rho_-<-n-1$; we omit the details.
\end{remark}

\begin{remark}\label{Rgam=0}
  We have in this section assumed $\gam>0$ in order to have good
  estimates of $F(s)$ as $|\Im s|\to\infty$ in the proofs.
It seems likely that the results can be extended to the case $\gam=0$
  too, under suitable conditions, but we have not pursued this beyond
  noting that the results above hold also for the examples in
  \refS{Sex} with $\gam=0$.

 For example, the uniform distribution in \refE{EU} has
  $\rho_+=\infty$, $\rho_-=-1$, $\gam=\gam'=0$, $\gk=0$ and a single,
  simple pole at $-1$; the series in \eqref{sumres} is thus 0 and the
  series in \eqref{xsumres} is 1, so \eqref{sumres} holds for  $x>e^\gk=1$
and \eqref{xsumres} holds for $x<e^\gk=1$. The asymptotic result in
  \refT{TAinfty} is not directly applicable, since the exponent
  $1/\gam=\infty$, but it can be interpreted as $f_X(x)=0$ for large $x$, which
  is correct.

The same holds, \emph{mutatis mutandis}, for the Pareto distribution
in \refE{EPareto}, where 
now there is a single pole at $\ga>0$ and the density vanishes on $(0,1)$.

Similarly, for the Beta distribution $\B(\ga,\gb)$ in \refE{Ebeta}, the
series in \eqref{xsumres} is
\begin{equation*}
  \sumn
  \frac{(-1)^n}{n!}\frac{\gG(\ga+\gb)}{\gG(\ga)\gG(\gb-n)}x^{n+\ga-1}
=
 \frac{\gG(\ga+\gb)}{\gG(\ga)\gG(\gb)}x^{\ga-1} 
\sumn \binom{\gb-1}n(-x)^{n}, 
\end{equation*}
which for $x<1=e^\gk$ converges to the density 
$f(x)=\bigpar{\xfrac{\gG(\ga+\gb)}{\gG(\ga)\gG(\gb)}}\cdot
\allowbreak x^{\ga-1}(1-x)^{\gb-1}$; 
for $x>1$ this series diverges unless $\gb$ is an integer (when the
series is finite but does not yield $f(x)=0$ for $x>1$), while
\eqref{sumres} holds trivially.

The $\gG(n)$ distribution in \refE{Egammamgf} is an example with a
multiple pole. There is a single pole at 1, with $c_n(1)=(-1)^n$ and
$c_\ell(1)=0$, $\ell\neq n$. Hence the sum in \eqref{tasum} is,
rewritten for $f_Y$ as stated in \refT{TAsum}, 
$(1/(n-1)!) y^{n-1} e^{-y}$, which is the correct density for $y>\gk=0$.
\end{remark}

We give some examples of applying the theorems above to the
distributions in \refS{Sex}. This is mainly as an illustration of the
theorems; we cannot expect to obtain any new results for these
classical distributions. Other applications of the theorems are given
in Theorems \ref{Tarea}, \ref{Tareaoo}, \ref{TMdensity}, \ref{TMoo},
\ref{Turn}, \ref{Turn2}, \ref{TDurn}, \ref{TDurn2}.

\begin{example}
  For the exponential distribution in \refE{Eexp}, \refT{TAinfty}
  yields $f(x)\sim e^{-x}$ as \xtoo{} (this is actually an identity
  for all $x>0$) and \refT{XTAsum} yields, since the poles are at
  $-n-1$, with $n=0,1,\dots$, $f(x)=\sumn(-1)^n x^n/n!$,
$x>0$, again a trival result.
\end{example}

\begin{example}
Similarly, for the Gumbel distribution in \refE{EGumbel}, 
$f(y)= e^{-y-e^{-y}}$ 
and the asymptotic formula in \refT{TAinfty0} is actually an
equality for all real $y$.
\end{example}

\begin{example}
  Consider the stable distribution in \refE{Estable} with $0<\ga<1$.
Since $\gam>0>\gam'$, we can apply \refT{TAsum}\ref{TAsum-}.
By \eqref{stab}, $F(s)\=\E S_\ga^s$ has simple poles at $s=n\ga$,
$n=1,2,\dots$, 
and, using \eqref{Asin},
\begin{equation}\label{stabres}
  \begin{split}
	\res{n\ga}(F)
&=\frac{-\ga\res{1-n}(\gG)}{\gG(1-n\ga)}
=\frac{(-1)^{n}\ga}{(n-1)!\,\gG(1-n\ga)}
\\&
=(-1)^n\frac{\ga\gG(n\ga)\sin(\pi n\ga)}{(n-1)!\,\pi}
=(-1)^n\frac{\gG(n\ga+1)\sin(\pi n\ga)}{\pi n!}.
  \end{split}
\end{equation}
(This includes the case when $n\ga$ is an integer, in which case
$n\ga$ is not a pole because of cancellation; \eqref{stabres}
then correctly yields $\res{n\ga}(F)=0$.)
We thus obtain
by \eqref{sumres}
\begin{equation}\label{fstab}
  f_{S_\ga}(x)
=\sumni (-1)^{n+1}\frac{\gG(n\ga+1)\sin(\pi n\ga)}{\pi n!} x^{-n\ga-1}.
\end{equation}
This is the well-known formula for the stable density, see
\citetq{XVII.(6.8)}{FellerII} (with $\gam=-\ga$ for the positive case
studied here).

In particular, as \xtoo, \eqref{fstab} or  \refT{TAfinite} 
(with $\rho_+=\ga$)
yields
\begin{equation*}
f_{S_\ga}(x)\sim -\res{\ga}(F)x^{-\ga-1} =
\frac{\ga}{\gG(1-\ga)}x^{-\ga-1},
\qquad x\to\infty.
\end{equation*}
As $x\to0$, \refT{TAinfty0} yields rapid convergence to 0:
\begin{equation}\label{fstaboo}
f_{S_\ga}(x)\sim 
C_3 x^{-(2-\ga)/(2-2\ga)} e^{-c_3 x^{-\ga/(1-\ga)}},
\qquad x\downto0,
  \end{equation}
with $c_3=(1-\ga)\ga^{\ga/(1-\ga)}$ and
$C_3=(2\pi(1-\ga))\qqw\ga^{1/(2-2\ga)}$. 
\end{example}

\begin{example}\label{EMLd}
  Consider the Mittag-Leffler distribution in \refE{EML} with $0<\ga<1$.
Since $\gam,\gam'>0$, we can apply \refT{XTAsum}\ref{XTAsum+}.
By \eqref{momML}, $F(s)\=\E M_\ga^s$ has simple poles at $s=-n$, $n=1,2,\dots$,
and, using \eqref{Asin},
\cf{} \eqref{stabres},
\begin{equation*}%\label{MLres}
  \begin{split}
	\res{-n}(F)
&=\frac{\res{1-n}(\gG)}{\gG(1-n\ga)}
=\frac{(-1)^{n-1}}{(n-1)!\,\gG(1-n\ga)}
%\\&
=(-1)^{n-1}\frac{\gG(n\ga)\sin(\pi n\ga)}{(n-1)!\,\pi}.
  \end{split}
\end{equation*}
(Again, this includes the case when $n\ga$ is an integer, in which case
$-n$ is not a pole but the formula correctly yields 0.)
We thus obtain
by \eqref{xsumres}
\begin{equation}\label{mlstab}
  \begin{split}
  f_{M_\ga}(x)
&=\sumni
(-1)^{n-1}\frac{\gG(n\ga)\sin(\pi n\ga)}{(n-1)!\,\pi}
x^{n-1}
\\&
=\summ
(-1)^{m}\frac{\gG(m\ga+\ga)\sin(\pi \ga(m+1))}{m!\,\pi}
x^{m}.	
  \end{split}
\end{equation}
(This is also easily obtained from the stable density
\eqref{fstab} since $M_\ga= S_\ga^{-\ga}$ and thus 
$f_{M_\ga}(x)=\ga\qw x^{-1/\ga-1}f_{S_\ga}(x^{-1/\ga})$.)
In particular, in accordance with \refT{T0},
\begin{equation*}
f_{M_\ga}(0+)=\frac{\gG(\ga)\sin(\pi\ga)}{\pi}=\frac1{\gG(1-\ga)}.  
\end{equation*}
As $x\to\infty$, \refT{TAinfty} yields
\begin{equation*}
  f_{M_\ga}(x)\sim C_2 x^{(2\ga-1)/(2-2\ga)} e^{-c_2x^{1/(1-\ga)}},
\qquad x\to\infty,
\end{equation*}
with $c_2=(1-\ga)\ga^{\ga/(1-\ga)}$ and 
$C_2=(2\pi(1-\ga))\qqw\ga^{(2\ga-1)/(2-2\ga)}$. 
(This also follows from \eqref{fstaboo}.) % and $M_\ga\eqd S_a^{-\ga}$.)
\end{example}

\begin{example}\label{ElevyD}
  The \Levy{} area in \refE{ELevy} has the \mgf{} \eqref{levymgf} with
  simple poles at $(n+\frac12)\pi$, $n\in\bbZ$. 
\refT{TAsum} yields, for $y>\gk=0$,
\begin{equation*}
  f(y)=\sumn(-1)^n e^{-(n+\frac12)\pi y}
=\frac{e^{-\pi y/2}}{1+e^{-\pi y}} =\frac1{2\cosh(\pi y/2)}
\end{equation*}
while 
\refT{XTAsum} yields, for $y<0$,
\begin{equation*}
  f(y)=\sumni(-1)^{n-1} e^{(n-\frac12)\pi y}
=\frac{e^{\pi y/2}}{1+e^{\pi y}} =\frac1{2\cosh(\pi y/2)}.
\end{equation*}
The two sums thus sum to the same analytic expression;
hence $A$ has the density $\xfrac1{2\cosh(\pi y/2)}$ for $-\infty<
y<\infty$.
(For a more elegant proof of this, see \eg{} 
\citetq{p. 91}{Protter}.) 
\end{example}

\section{Brownian supremum process area}\label{Sarea}

We consider the integral $\cA=\cA(1)$ of the Brownian supremum process
defined in \eqref{a1}.

\begin{remark}\label{Rlocal}
Let $L(t)$ denote the local time of $B(t)$ at 0.
%(This is a measure of the time $B(t)$ spends at 0 in the interval
%$[0,t]$, see Revuz and Yor \cite[Chapter VI]{RY} for details. 
It is well-known that the processes $S(t)$ and $L(t)$, have the same
distribution \cite[Chapter VI.2]{RY}:
\begin{equation*}
\bigset{S(t)}_{t\ge0} \eqd \bigset{L(t)}_{t\ge0}.
\end{equation*}
Consequently, $\cA(\taux)\eqd\intotau L(t)\dd t$, 
so we obtain the same results for this integral.
\end{remark}

Let $\psi$  denote the Laplace transform of $\cA$:
\begin{equation}
  \label{psi}
\psi(s)\=\E e^ {-s\cA}.
\end{equation}
\citet{SJ202} proved the
following formula for the Laplace transform of a variation of
$\psi$, 
or in other words, a \emph{double} Laplace transform of
$\cA$:
For all $\ga,\gl>0$,
\begin{equation}\label{t1}
\intoo \psi\bigpar{\ga t\qqc} e^{-\gl t} \dd t 
= \intoo\Bigpar{1+\frac{3 \ga t}{\sqrt{8 \gl}}}^{-2/3} e^{-\gl t} \dd t. 
\end{equation}

\citet{SJ202} used \eqref{t1} to compute the integer moments
$\E\cA^n$, $n\in\bbN$; \refT{T2} extends their formula to all real and
complex moments.

\begin{proof} [Proof of \refT{T2}]
Consider for convenience $X\=\frac{\sqrt8}{3}\cA$. Taking
$\ga=\sqrt8/3$ in \eqref{t1}, we find, for $\gl>0$,
\begin{equation}\label{t2a}
\intoo \E e^{-t\qqc X} e^{-\gl t} \dd t 
= \intoo\Bigpar{1+\frac{ t}{\sqrt{ \gl}}}^{-2/3} e^{-\gl t} \dd t. 
\end{equation}
Denote the common value of the integrals in \eqref{t2a} by $G(\gl)$.

Let $-1<s<0$, and integrate $\gl^sG(\gl)$. From the \lhs{} in
\eqref{t2a} we obtain, using Fubini's theorem a couple of times, 
the standard Gamma integral \eqref{Agamma1}, and the change of
variables $t\qqc=u$, 
\begin{equation}\label{t2l}
  \begin{split}
\intoo G(\gl)\gl^s\dd\gl
&=
\intoo\intoo\E e^{-t\qqc X} e^{-\gl t} \gl^s\dd t \dd\gl	
\\&
=
\gG(s+1)\E\intoo e^{-t\qqc X} t^{-s-1}\dd t 
\\&
=
\gG(s+1)\E\frac23\intoo e^{-u X} u^{-2s/3-1}\dd u 
\\&
=
\frac23\gG(s+1)\gG(-2s/3)\E X^{2s/3}. 
  \end{split}
\end{equation}

Similarly, from the \rhs\ of \eqref{t2a}, 
using the changes of variables $t=\gl\qq x$ 
and $\gl=u\qqqb$,
and the standard Gamma and Beta
integrals \eqref{Agamma1} and \eqref{Abeta2}, 
still assuming $-1<s<0$,
\begin{equation}\label{t2r}
  \begin{split}
\intoo G(\gl)\gl^s\dd\gl
&= \intoo\intoo\Bigpar{1+\frac{ t}{\sqrt{ \gl}}}^{-2/3} e^{-\gl t}
\gl^s \dd t \dd\gl \\
&= \intoo\intoo\lrpar{1+x}^{-2/3} e^{-\gl\qqc x}
\gl^{s+\xfrac12} \dd x \dd\gl \\
&= \frac23 \intoo\intoo\lrpar{1+x}^{-2/3} e^{-u x}
u^{2s/3} \dd x \dd u \\
&= \frac23 \Gamma(2s/3+1)\intoo\lrpar{1+x}^{-2/3}  x^{-2s/3-1} \dd x  \\
&= \frac23 \Gamma(2s/3+1)\frac{\Gamma(-2s/3)\Gamma(2s/3+2/3)}{\gG(2/3)}.
  \end{split}
\end{equation}

Setting the \rhs{s} of \eqref{t2l} and \eqref{t2r} equal, we find
after some cancellations,
\begin{equation*}
\E X^{2s/3}  
=  \frac{\Gamma(2s/3+1)\Gamma(2s/3+2/3)}{\gG(2/3)\gG(s+1)},
\end{equation*}
for $-1<s<0$, and thus, replacing $s$ by $3s/2$,
\begin{equation}\label{t2e}
\E X^{s}  
=  \frac{\Gamma(s+1)\Gamma(s+2/3)}{\gG(2/3)\gG(3s/2+1)},
\end{equation}
for $-2/3 < s <0$. 
%In particular, $\E X^s<\infty$ for such $s$.

We use \refT{T1} to extend the domain of validity of
\eqref{t2e}. Denote, as usual, the \rhs\ of \eqref{t2e} by $F(s)$, and
note that $s=-2/3$ is \emph{not} a pole of $F(s)$; it is a removable
singularity since the poles in the numerator and denominator at $-2/3$
cancel. The first pole of $G(s)$ on the negative real axis is $s=-1$.
This can also be seen by the functional equation $\gG(z+1)=z\gG(z)$, 
which enables us to rewrite \eqref{t2e} as 
\begin{equation}\label{t2f}
\Bigparfrac{\sqrt8}3^s \E\cA^{s}  
=
\E X^{s}  
=  \frac{\Gamma(s+1)\Gamma(s+5/3)}{\gG(5/3)\gG(3s/2+2)},
\end{equation}
where the right hand side clearly has a pole at $-1$ but not in $(-1,\infty)$.
Hence 
$\rho_-=-1$ and $\rho_+=\infty$, so \eqref{t2e} and
\eqref{t2f} hold for $\Re s>-1$ by \refT{T1}, while $\E\cA^s=\infty$
for $s\le-1$.

Next, the triplication and duplication formulas \eqref{Atriple} and
\eqref{Adouble} yield
\begin{multline*}
  \gG(s+1)\gG(s+4/3)\gG(s+5/3)
=2\pi 3^{-3s-3+1/2}\gG(3s+3)\\
=\pi\qq 2^{3s+3} 3^{-3s-5/2}\gG(3s/2+3/2)\gG(3s/2+2),	
\end{multline*}
and thus
\begin{equation*}
\frac{  \gG(s+1)\gG(s+5/3)}{ \gG(3s/2+2)}
=\pi\qq 2^{3} 3^{-5/2}\frac{\gG(3s/2+3/2) }{ \gG(s+4/3)}
\Bigparfrac {2^3}{3^3}^s.
\end{equation*}
The third formula for $\E\cA^s$  follows by substituting this into
\eqref{t2f}, and using $\gG(1/3)\gG(2/3)=\pi/\sin(\pi/3)=2\pi/\sqrt3$
from \eqref{Asin}. (The constant factor can always be found by setting
$s=0$, see \refR{RC}.)

Similarly, the final formula follows by applying 
the duplication formula \eqref{Adouble} to $\gG(s+4/3)$
and the triplication formula \eqref{Atriple} to $\gG(3s/2+3/2)$.
(Alternatively, we may use \eqref{t2e}, applying the duplication
formula to $\gG(s+1)$ and $\gG(s+2/3)$ and the triplication formula to
$\gG(3s/2+1)$.)
\end{proof}

We have $\rho_+=\infty$ and $\rho_-=-1$ for $\cA$. Further,
the parameters in \eqref{gam}--\eqref{CC1} are, from any of the
expressions in \refT{T2}:
$\gam=\gam'=1/2$, $\gd=1/6$, $\gk=-\frac12\log3$,
$C_1=\pi\qqw\gG(1/3)=2\sqrt{\pi/3}/\gG(2/3)$. 
\refT{TRE} thus yields
\begin{equation}
  \label{amom}
\E \cA^s\sim
%\pi\qqw\gG(1/3) 
\frac{\gG(1/3)}{\pi\qq}s^{1/6} e^{\frac12s\log s-(\frac12\log 3+\frac12)s}
= \frac{\gG(1/3)}{\pi\qq}s^{1/6}\Bigparfrac{s}{3e}^{s/2}
%,
%\qquad{s\to\infty},
\end{equation}
as $s\to\infty$,
found for integer $s$ in \cite{SJ202}.

\begin{proof}[Proof of \refT{Tarea}]
The existence of the density function $f_\cA(x)$ follows from
\refT{TF}. The explicit formulas are obtained from
\refT{XTAsum} as follows.
 We use the last expression in \refT{T2} for $F(s)\=\E\cA^s$, where
  there is no  cancellation of poles. The poles are thus given by, 
for $n\in\bbZgeo$, 
$s/2+1/2=-n$ and
  $s/2+5/6=-n$,  \ie{} $s=-2n-1$ and $s=-2n-5/3$; all poles are simple.
The same formula yields the residues, using \eqref{Ares} and \eqref{Asin},
\begin{align*}
  \Res_{-2n-1}(F)
&=
\frac{\Gamma(1/3)}{2\qqq\pi}\xdot
2\frac{(-1)^n}{n!}\frac{\gG(1/3-n) }{ \gG(1/6-n)}
\xdot\Bigparfrac {2}{3}^{-n-1/2}
\\
&=
(-1)^n\frac{2^{1/6}3\qq\Gamma(1/3)}{\pi}\xdot
\frac{\gG(n+5/6) }{n!\, \gG(n+2/3)}
\xdot\frac{\sin(\pi/6-n\pi)}{\sin(\pi/3-n\pi)}
\xdot\Bigparfrac {3}{2}^{n}
\\
&=
(-1)^n\frac{2^{1/6}\Gamma(1/3)}{\pi}\xdot
\frac{\gG(n+5/6) }{n!\, \gG(n+2/3)}
\xdot\Bigparfrac {3}{2}^{n},
\\
  \Res_{-2n-5/3}(F)
&=
\frac{\Gamma(1/3)}{2\qqq\pi}\xdot
2\frac{(-1)^n}{n!}\frac{\gG(-1/3-n) }{ \gG(-1/6-n)}
\xdot\Bigparfrac {2}{3}^{-n-5/6}
\\
&=
(-1)^n\frac{3\qqq\Gamma(1/3)}{2^{1/6}\pi}\xdot
\frac{\gG(n+7/6) }{n!\, \gG(n+4/3)}
\xdot\Bigparfrac {3}{2}^{n}.
\end{align*}
Consequently, by \refT{XTAsum}, in particular \eqref{xsumres},
\begin{equation*}
  \begin{split}
f_\cA(x)
&=\frac{2^{1/6}\Gamma(1/3)}{\pi}
\sumn (-1)^n
\frac{\gG(n+5/6) }{n!\, \gG(n+2/3)}
\,\Bigparfrac {3}{2}^{n} x^{2n}
\\
&\qquad
{}
+
\frac{3\qqq\Gamma(1/3)}{2^{1/6}\pi}
\sumn(-1)^n
\frac{\gG(n+7/6) }{n!\, \gG(n+4/3)}
\,\Bigparfrac {3}{2}^{n}
x^{2n+2/3}.
  \end{split}
\end{equation*}
By the definition of $\Fii$, 
and simplifying the constants  using \eqref{Adouble} and \eqref{Asin},
this can be written as 
\begin{equation*}
f_\cA(x)=
\frac{2\qq}{\pi\qq}\Fii\lrxpar{\frac56;\frac23;-\frac32 x^2}
+ \frac{2^{-1/6}3\qqq}{\gG(5/6)}\,x^{2/3}
  \Fii\lrxpar{\frac76;\frac43;-\frac32 x^2}.
\end{equation*}
By Kummer's transformation \cite[(13.1.27)]{AS}, this equals
\begin{equation*}
e^{-\frac32 x^2}\lrpar{
\frac{2\qq}{\pi\qq}\Fii\lrxpar{-\frac16;\frac23;\frac32 x^2}
+ \frac{2^{-1/6}3\qqq}{\gG(5/6)}\,x^{2/3}
  \Fii\lrxpar{\frac16;\frac43;\frac32 x^2}
},
\end{equation*}
which can be rewritten as the two last formulas in the theorem
by the definition of $U$ \cite[(13.1.3)]{AS}, 
see also \cite[(13.1.29)]{AS}, 
again using \eqref{Adouble} and \eqref{Asin} to simplify constants.
\end{proof}

\begin{proof}[Proof of \refT{Tareaoo}]
  Immediate by \refT{TAinfty}.
\end{proof}

\begin{remark}
  The hypergeometric function $\Fii(a;b;x)$  and $U(a;b;x)$ 
satisfies Kummer's
  equation 
$x F''+(b-x)F'-aF=0$
\cite[(13.1.1)]{AS}, 
and it follows easily that
$f_\cA$ satisfies the differential equation
  \begin{equation}
xf_\cA''(x)+(3x^2+\tfrac13)f_\cA'(x)+5xf_\cA(x)=0,
\qquad x>0.
  \end{equation}
We guess that it also is possible to derive this equation directly
from \eqref{t1} by manipulations of Laplace transforms, but we have
not pursued this.
\end{remark}

\section{A hashing variable}\label{Shash}

As said in \refS{S:intro}, when studying the maximum displacement in
hashing with linear probing, \citetq{Theorem 5.1}{PeterssonI} found as
a limit a random variable  
$\cM$ with
  the distribution %\cite[Theorem 5.1]{PeterssonI}
  \begin{equation}
	\label{np1}
\P(\cM> x) = \psi(x\qqc)
=\E e^ {-x\qqc\cA}, \qquad x>0,
  \end{equation}
where $\cA$ is the Brownian supremum area studied in \refS{Sarea}.
\refL{LVZ} shows that this type of relation preserves moments \ogt;
hence  $\cM$ has moments \ogt, but
we have postponed the proof until now.

\begin{proof}[Proof of \refL{LVZ}]
For $x\ge0$, 
\begin{equation*}
  \begin{split}
  \P(T^{1/\ga}/Z^{1/\ga}>x)=\P(T>x^\ga Z)
=\E\bigpar{ \P(T>x^\ga Z\mid Z)}
=\E e^{-x^\ga Z}
  \end{split}
\end{equation*}
which shows that 
\eqref{vz1} and \eqref{vz2} are equivalent.

If \eqref{vz1} or \eqref{vz2} holds, and thus both hold, then, for $s>-\ga$,
using \eqref{Exp},
\begin{equation*}
  \E V^s=\E T^{s/\ga}\E Z^{-s/\ga}
=\Gamma(s/\ga+1)\E Z^{-s/\ga}.
\qedhere
\end{equation*}
\end{proof}

\begin{proof}[Proof of \refT{TM}]
  By \eqref{np1} and \refL{LVZ}, with $\ga=3/2$,
  \begin{equation*}
	\E \cM^{s}=\gG(2s/3+1)\E\cA^{-2s/3},
  \end{equation*}
and the result follows from \refT{T2}; for the last formula we also
use \eqref{Adouble}.
\end{proof}

Note also that \refL{LVZ} yields the representation
\begin{equation}\label{mxa}
  \cM\eqd T\qqqb\cA\qqqbw,
\end{equation}
where $T\qqqb$ has a Weibull distribution with parameter $3/2$, \cf{}
\refE{EWeibull}, and is independent of $\cA$.

For $\cM$, the parameters in \eqref{gam}--\eqref{CC1} are, from any of the
expressions in \refT{TM}:
$\gam=1$, $\gam'=1/3$, $\gd=2/3$, $\gk=\frac13\log2$,
$C_1=2^{7/6}3^{-2/3}\gG(1/3)=2^{13/6}3^{-7/6}\pi/\gG(2/3)$. 
Furthermore, the function $F(s)\=\E\cM^s$ (extended to all of the
complex plane)  
has residue $-3/\sqrt{2\pi}$ at $\rho_+=3/2$, and $\sqrt{2/\pi}$
at $\rho_-=-3/2$. (See the proofs below for the other residues.)

\begin{proof}[Proof of \refT{TMdensity}]
As in the proof of \refT{Tarea}, 
the existence of the density function $f_\cM(x)$ follows from
\refT{TF} and the explicit formulas are obtained from
\refT{XTAsum}.
 We use the third expression in \refT{T2} for $F(s)\=\E\cM^s$.
The poles $\rho<0$ all come from the factor $\gG(1+2s/3)$ and are thus
  given by,  for $n\in\bbZgeo$, 
$1+2\rho/3=-n$,  \ie{} $\rho=-3n/2-3/2$. All poles are simple
and we find using \eqref{Ares} the residues
\begin{multline*}
\res{-\tfrac32n-\tfrac32} (F)
\\
= \frac{\Gamma(1/3)}{2\qqq\pi}
\xdot
\frac{\gG(1+n/2)\,\gG(4/3+n/2) }{ \gG(7/6+n/2)}
\xdot\Bigparfrac {3}{2}^{-1/2-n/2}
\xdot\frac32\frac{(-1)^n}{n!}.
\end{multline*}
Consequently, by \refT{XTAsum}, in particular \eqref{xsumres},
\begin{equation*}
  \begin{split}
f_\cM(x)
%&=\sumn \res{s=-\tfrac32n-\tfrac32} (\E \cM^s)\, x^{3n/2+1/2}
%\\
&= \frac{3\qq\Gamma(1/3)}{2^{5/6}\pi}
\sumn (-1)^n
\frac{\gG(1+n/2)\,\gG(4/3+n/2) }{ \gG(7/6+n/2)\,{n!}}
\Bigparfrac {2}{3}^{n/2}
 x^{3n/2+1/2}.
  \end{split}
\end{equation*}
Splitting the sum into two parts, for $n=2k$ and $n=2k+1$, and using 
$(2k)!=\pi\qqw 2^{2k}\gG(k+1/2)\,k!$
and
$(2k+1)!=\pi\qqw 2^{2k+1}k!\,\gG(k+3/2)$, both instances of \eqref{Adouble},
we obtain, as usual using \eqref{Adouble} and \eqref{Asin},
\begin{equation*}
  \begin{split}
f_\cM(x)
&= \frac{3\qq\Gamma(1/3)}{2^{5/6}\pi\qq}
\sumk
\frac{\gG(4/3+k) }{ \gG(7/6+k)\,\gG(1/2+k)}
\Bigparfrac {1}{6}^{k}
 x^{3k+1/2}
\\
&\qquad{}- \frac{\Gamma(1/3)}{2^{4/3}\pi\qq}
\sumk
\frac{\gG(11/6+k) }{ \gG(5/3+k)\,{k!}}
\Bigparfrac {1}{6}^{k}
 x^{3k+2}
\\
&= \frac{3\qq\Gamma(1/3)}{2^{5/6}\pi\qq}
\frac{\gG(4/3) }{ \gG(7/6)\,\gG(1/2)}
x^{1/2}
\FF22\Bigpar{\frac43,1;\frac76,\frac12;\frac{x^3}6 }
\\
&\qquad{}- \frac{\Gamma(1/3)}{2^{4/3}\pi\qq}
\frac{\gG(11/6) }{ \gG(5/3)}
x^2
\Fii\Bigpar{\frac{11}6;\frac53;
\frac {x^3}{6}}
\\
&= \frac{2\qq}{\pi\qq}
x^{1/2}
\FF22\Bigpar{\frac43,1;\frac76,\frac12;\frac{x^3}6 }
- \frac58
x^2
\Fii\Bigpar{\frac{11}6;\frac53;
\frac {x^3}{6}}.
\qedhere
  \end{split}
%\qedhere
\end{equation*}
\end{proof}

\begin{proof}[Proof of \refT{TMoo}]
As remarked above, the residue at $\rho_+=3/2$ is $-3/\sqrt{2\pi}$,
and the next pole is at $5/2$, which yields \eqref{laban} by \refT{TAfinite}.

More precisely, by the last expression in \refT{TM}, 
there are, 
on the positive real axis,
%in the right half-plane, 
poles when $1/2-s/3=-n$ or $5/6-s/3=-n$ for integer $n\ge0$, \ie,
$s=3n+3/2$ and $s=3n+5/2$.
The residues are, using \eqref{Ares} and \eqref{Asin},
\begin{align*}
  \Res_{3n+3/2}(F)
&=
-3\frac{(-1)^n}{n!}
\frac{\Gamma(1/3)}{2\qqq\pi\qqc}\xdot
\frac{\gG(1+n)\gG(3/2+n)\gG(1/3-n) }{ \gG(1/6-n)}
\xdot6^{n+1/2}
\\
&=
(-1)^{n+1}\frac{2^{1/6}3\,\Gamma(1/3)}{\pi\qqc}\xdot
\frac{\gG(3/2+n)\gG(5/6+n) }{\gG(2/3+n)}
\xdot6^{n},
\\
  \Res_{3n+5/2}(F)
&=
-3\frac{(-1)^n}{n!}
\frac{\Gamma(1/3)}{2\qqq\pi\qqc}\xdot
\frac{\gG(4/3+n)\,\gG(11/6+n)\,\gG(-1/3-n) }{ \gG(-1/6-n)}
\xdot6^{n+5/6}
\\
&=
(-1)^{n+1}\frac{2\qq3^{4/3}\Gamma(1/3)}{\pi\qqc}\xdot
\frac{\gG(11/6+n)\,\gG(7/6+n) }{n!}
\xdot6^{n}.
\end{align*}
By \refT{TAfinite}, there is an asymptotic expansion
$-\sum_{\rho>0}\res\rho(F)x^{-\rho-1}$, which by the definition of the
(generalized) hypergeometric series can be written
as in \eqref{spoke}, yet again  using \eqref{Ares} and \eqref{Asin}.
\end{proof}

\section{Triangular and diagonal \Polya{} urns}\label{Surn}

A generalized \Polya{} urn contains balls of several different
colours. At each time $n\ge1$, one of the balls is drawn at random,
and a set of new balls, depending on the colour of the drawn ball, is
added to the urn. We consider for simplicity only the case of two
colours, say black and white; the replacement rule may then be
described by a matrix $\smatrixx{a&b\\c&d}$, meaning that if the drawn
ball is black [white], it is replaced together with 
$a$ black and $b$ white balls  [$c$ black and $d$ white balls].
It is here natural to let $a,b,c,d$ be non-negative integers, but in
fact, the model can be defined (and the results below hold)
for arbitrary real $a,b,c,d\ge0$,
see \cite{SJ154,SJ169}.
(Further, under certain conditions some of the entries can be negative
too, 
but that case is not interesting here.) 
Different values of the parameters yield a variety
of different limit laws for the numbers $B_n$ and $W_n$ of black and
white balls in the urn 
after $n$ steps, see \eg{}
\cite{F:exact,SJ154,SJ169} and the references given there. We are here
interested in the special case of a \emph{triangular urn}, meaning
that the replacement matrix is triangular, say $b=0$.
We start with $B_0=b_0\ge0$ black and $W_0=w_0\ge0$ white balls, and assume
$w_0>0$ (otherwise, there will never be any white balls).

\subsection{Balanced triangular urns}
Assume that the urn is triangular and \emph{balanced}, meaning that
the total number 
of added balls does not depend on the drawn ball, \ie, $a=c+d$; we
further assume that $a,c,d>0$; thus $a>d>0$ and $c=a-d$.
In this case, it is shown by 
\citet{Puy},
\citetq{Section 7}{F:exact}  and (with a
different proof) \citetq{Theorems  1.3(v) and 1.7}{SJ169}
that $W_n/n^{d/a}\dto W$ for a random variable $W$ with moments \ogt{}
given by  
\begin{equation}\label{urn}
  \E W^s=d^s \frac{\gG((b_0+w_0)/a)}{\gG(w_0/d)}
\xdot\frac{\gG(s+w_0/d)}{\gG(ds/a+(b_0+w_0)/a)},
\qquad \Re s >-\frac{w_0}d.
\end{equation}
In the special case $(b_0,w_0)=(c,d)$, and thus $b_0+w_0=a$, this
simplifies to $d^s \gG(s+1)/\gG(ds/a+1)$, so $W/d$ has a
Mittag-Leffler distribution with parameter $d/a\in(0,1)$, see
\eqref{momML}.

All poles of $F(s)\=\E W^s$ are on the negative real axis, so $\rho_+=\infty$.
In general, \eqref{urn} 
shows that there is a pole at $-w_0/d$, but if $b_0=0$, then this singularity
is removable and the first pole on the negative real axis is
$-w_0/d-1$.
We thus have
%$\rho_+=\infty$ and  
$\rho_-=-w_0/d$ when $b_0>0$, but
$\rho_-=-w_0/d-1$ when $b_0=0$.
In fact, if $b_0=0$, so we start with only $w_0$ white balls, the
first drawn ball is necessarily white, and thus urn after the first
draw contains $c$ black and $w_0+d$ white balls. Thus the limit random
variable $W$ is the same for the initial conditions $(0,w_0)$ and
$(c,w_0+d)$, and we may without loss of generality assume that $b_0>0$.

By \eqref{urn}, we have
%$\rho_+=\infty$, $\rho_-=-w_0/d$, 
$\gam=\gam'=1-d/a=c/a$, 
$\gd=w_0/d-(b_0+w_0)/a$, 
$\gk=-\frac da\log\frac da+\log d
= (c\log d+d\log a)/a$, and
$C_1=(a/d)^{(b_0+w_0)/a-1/2}\xfrac{\gG((b_0+w_0)/a)}{\gG(w_0/d)}$.

The function $F(s)$ in \eqref{urn} has simple poles at $s=-w_0/d-n$,
$n=0,1,\dots$, (except that some of these may in fact be
removable singularities) and Theorems \ref{TF} and \ref{XTAsum} 
yield by a straightforward calculation of the residues, using
\eqref{Ares} and \eqref{Asin}, the following:

\begin{theorem}
  \label{Turn}
The limit variable $W$ for a balanced triangular urn
$\smatrixx{a&0\\c&d}$ with $a=c+d$ and $a,c,d,w_0>0$
has a  density function $f_W$ on $\ooo$ given by, for $x>0$,
\begingroup\multlinegap=0pt
  \begin{multline*}
  f_W(x)= \frac{\gG((b_0+w_0)/a)}{\gG(w_0/d)}\sumn    
\frac{(-1)^n}{n!}
\xdot\frac{d^{-n-w_0/d}}{\gG(-dn/a+b_0/a)}\,
x^{n+w_0/d-1}
\\
= \frac{\gG((b_0+w_0)/a)}{\pi d\gG(w_0/d)}\sumn    
\frac{(-1)^n}{n!}
\,\gG\Bigpar{\frac{dn+a-b_0}a}
\sin\frac{\pi(b_0-dn)}{a}
\Bigparfrac xd^{n+w_0/d-1}.	
  \end{multline*}
  \endgroup
\end{theorem}
In fact, \cite{F:exact} even gives a local limit theorem to this
density function.

\begin{remark}
  It follows from \refT{Turn} by comparison with \refE{EMLd}, 
or more simply directly from \eqref{urn} and \eqref{momML},
that in the special case $b_0=0$,
$W/d$ has a Mittag-Leffler($d/a$) distribution conjugated with
$x^{w_0/d}$,
see \refR{Rconj};
similarly, in the special case $b_0=c=a-d$, 
$W/d$ has a Mittag-Leffler($d/a$) distribution conjugated with $x^{(w_0-d)/d}$.
\end{remark}

\refT{Turn} shows immediately that as $x\downto0$, the density
$f_W(x)$ 
satisfies $f_W(x)\sim C' x^{w_0/d-1}$ where
$C'=\gG((b_0+w_0)/a)\allowbreak(\gG(w_0/d)\gG(b_0/a))\qw
d^{-w_0/d}>0$, 
provided $b_0>0$.
For large $x$, \refT{TAinfty} yields:
\begin{theorem}
  \label{Turn2}
As \xtoo,
\begin{equation*}
  f_W(x)\sim C_2 x^{c_1-1}e^{-c_2 x^{a/c}},
\end{equation*}
with
$c_1=(\gd+1/2)a/c$,
$c_2=ca^{-a/c}d\qw$, $C_2=C_1(2\pi c/a)\qqw (da^{d/c})^{-(\gd+1/2)}$,
where $\gd$ and $C_1$ are given above.
\end{theorem}

\begin{remark}
  For non-balanced triangular urns ($a\neq c+d$), limit results are
  given in \cite{SJ169}, but the results are more complicated and we
  do not believe that the limits have moments \ogt.
(See for example \cite[Theorem 1.6]{SJ169}, which gives a complicated
  integral formula for the moments in the case $a>d>0$, $c>0$. In the
  balanced case, it  simplifies to \eqref{urn}, but as far as we know,
  there is no similar simplification in general.
\end{remark}

\begin{remark}
  The case of triangular urns with three or more colours is not yet
  fully explored. 
Limit laws with \mogt{} occur in
  some cases, but presumably not in all.
Some such results are given by \citet{Puy}, see also \citet{F:exact}.
\end{remark}

\subsection{Diagonal urns}\label{SSdiagonal}
In the diagonal case $b=c=0$ (with $a,d,b_0,x_0>0$ to avoid
trivialities), 
there are simple limit
results, see \cite[Theorem 1.4]{SJ169}. We distinguish between three cases.

\begin{xenumerate}
  \item
If $a=d$, the classical \Polya{} urn \cite{EggPol,Polya},
$W_n/n\to W$ where $W/d\simin \B(w_0/a,b_0/a)$. Hence, by \refE{Ebeta},
$W$ has \mogt
\begin{equation}
  \E W^s = a^s 
\frac{\gG((b_0+w_0)/a)}{\gG(w_0/a)}
\xdot
\frac{\gG(s+w_0/a)}{\gG(s+(b_0+w_0)/a)},
\qquad s>-\frac{w_0}{a}.
\end{equation}
Hence, recalling that $a=d$,
\eqref{urn} holds in this case too.
We have
$\rho_+=\infty$, $\rho_-=-w_0/a$, 
$\gam=\gam'=0$,
$\gd=-b_0/a$, 
$\gk=\log a$,
$C_1=\xfrac{\gG((b_0+w_0)/a)}{\gG(w_0/a)}$.

\item
If $a>d$, $W_n/n^{d/a}\dto W\=d U^{-d/a}V$ where $U\simin\gG(b_0/a)$
and $V\simin\gG(w_0/d)$ are independent. Thus, by \eqref{Gamma},
\begin{equation}\label{diag}
  \E W^s= d^s \,\frac{\gG(b_0/a-ds/a)\,\gG(w_0/d+s)}{\gG(b_0/a)\,\gG(w_0/d)},
\qquad -\frac{w_0}d<\Re s< \frac{b_0}d.
\end{equation}
We have 
$\rho_+=b_0/d$, $\rho_-=-w_0/d$,
$\gam=1+d/a$, $\gam'=1-d/a$, $\gd=b_0/a+w_0/d-1$, 
$\gk=-\frac da\log\frac da+\log d$,
$C_1=2\pi(\gG(b_0/a)\gG(w_0/d))\qw(d/a)^{b_0/a-1/2}$. 
Theorems \refand{TF}{XTAsum} apply again and
yield the following:

\begin{theorem}
  \label{TDurn}
The limit variable $W$ for a diagonal urn
$\smatrixx{a&0\\0&d}$ with $a>d>0$ and $b_0,w_0>0$
has a  density function $f_W$ on $\ooo$ given by, for $x>0$,
\begin{equation*}
  \begin{split}
  f_W(x)&= \frac1{d\gG(b_0/a)\gG(w_0/d)}\sumn    
\frac{(-1)^n}{n!}
\,\gG\Bigpar{\frac{dn+b_0+w_0}a}
\Bigparfrac xd^{n+w_0/d-1}.	
  \end{split}
\end{equation*}
\end{theorem}

Again, the asymptotic $f_W(x)\sim C' x^{w_0/d-1}$ 
as $x\downto0$, for some $C'>0$, is immediate.
For large $x$, we this time use \refT{TAfinite}, since
$\rho_+<\infty$.
The poles of \eqref{diag} on the positive real axis are $(an+b_0)/d$,
$n=0,1,\dots$, and the residues are easily calculated.
This yields a divergent asymptotic expansion, interpreted as in \refR{RF20}.

\begin{theorem}
  \label{TDurn2}
%The limit variable $W$ for a diagonal urn
%has a  density function $f_W$ on $\ooo$ given by, for $x>0$,
As \xtoo, the density $f_W(x)$ has an  asymptotic expansion
\begin{equation*}
  \begin{split}
  f_W(x)&\sim \frac a{d^2\gG(b_0/a)\gG(w_0/d)}\sumn    
\frac{(-1)^n}{n!}
\,\gG\Bigpar{\frac{an+b_0+w_0}d}
\Bigparfrac xd^{-an/d-b_0/d-1}.	
  \end{split}
\end{equation*}
\end{theorem}

\item
If $a<d$, we may interchange the two colours and obtain
$n^{-a/d}(nd-W_n)\dto W\=d U V^{-a/d}$, with $U$ and $V$ as above, and
$aW/d$ has the distribution in (ii) with the exchanges $a\leftrightarrow d$
and $b_0\leftrightarrow w_0$.
\end{xenumerate}

\section{Further examples}\label{Sex2}

We give a couple of further examples, or rather counter examples.
\begin{example}\label{Ecounter}
Let $X$ have a distribution that is a mixture of a point mass at 1 and
a uniform distribution on $\oi$, with equal weights; thus
$X=1-V+VU$ where $V\simin\Be(1/2)$ and $U\simin \U(0,1)$ are independent.
Then, for $\Re s>-1$,
\begin{equation}
  \begin{split}
\E X^s 
&=	
\frac12\cdot 1^s+\frac12\cdot\E U^s
=	
\frac12+\frac12\,\frac1{s+1}
=\frac{s+2}{2(s+1)}
=\frac{\gG(s+3)\,\gG(s+1)}{2\,\gG(s+2)^2}.
  \end{split}
\end{equation}
Equivalently,
\begin{equation}
\E X^s 
=\frac{s/2+1}{s+1}
=\frac{\gG(s/2+2)\,\gG(s+1)}{\gG(s/2+1)\,\gG(s+2)}.
\end{equation}
Hence $X$ has moments \ogt.
We have $\rho_+=\infty$, $\rho_-=-1$, $\gam=\gam'=\gd=\gk=0$ and $C_1=1/2$.
Note that $\E X^{\ii t}\to1/2\neq0$ as $t\to\pm\infty$;
\cf{} \refR{R00}.
\end{example}

\begin{example}
  \label{ETAsum}
Consider $X\=T/U$, where $T\sim\Exp(1)$ and $U\sim \U(0,1)$ are
independent. Then, see \refR{Rclosure} and Examples \refand{Eexp}{EU}, $X$
has \mogt
\begin{equation}
  \E X^s = \frac{\gG(s+1)}{1-s}
=\frac{\gG(s+1)\,\gG(1-s)}{\gG(2-s)},
\qquad -1<\Re s<1.
\end{equation}
Consequently, $\rho_+=1$ and $\rho_-=-1$. 
There is an infinite number
of poles on the negative real axis, \viz{} $-1,-2,\dots$,
but the only pole on the positive real axis is 1.
We have $\gam=\gam'=1$, $\gd=-1/2$, $\gk=0$, $C_1=\sqrt{2\pi}$.

It is easy to find the density of $X=T/U$: for $x>0$,
\begin{equation*}
  \P(T/U>x) =\P(T>Ux)
=\intoi\P(T>ux)\dd u
=\intoi e^{-ux}\dd u
=\frac{1-e^{-x}}x
\end{equation*}
and thus $X$ has the density function
\begin{equation}\label{etaf}
  f(x)=-\frac{\dd}{\dd x}\frac{1-e^{-x}}x
=\frac{1-(1+x)e^{-x}}{x^2},
\qquad x>0.
\end{equation}

Since $\gam,\gam'>0$, \refT{XTAsum} applies. The residue at $-n-1$ is
$(-1)^n/(n!\,(n+2)$ and \eqref{xsumres} yields  
\begin{equation*}
  f(x)=\sumn \frac{(-1)^n }{(n+2) n!} x^n
\end{equation*}
which, of course, also follows directly from \eqref{etaf}.

However, in \refT{TAsum}, although the sum in \eqref{sumres} consists
of a single term $x\qww$ 
and thus converges, the sum $x\qww\neq f(x)$ for all $x>0$, as
asserted in \refR{RTAsumdiv}. (But the error is exponentially small,
and the estimates in \refT{TAfinite} apply.)
\end{example}

\section{Further remarks}\label{Sfurther}

\begin{remark}
Suppose that $X$ is a positive random variable with finite moments (of
  all positive orders): $\E X^n<\infty$ for $n\ge0$.
If $X$ has moments \ogt, then $\rho_+=\infty$ and \eqref{gamma} gives,
  in particular,   a formula for all integer moments $\E X^n$ in terms
  of Gamma functions.
However, the converse does not hold; even if \eqref{gamma} holds for
every integer $s\ge0$, it does not necessarily hold for other $s$.
An example is provided by Stieltjes' original example of indeterminacy
  in the moment problem \cite[\S 55]{Stieltjes}: Let, for
  $\gl\in[-1,1]$, 
$X_\gl$ have the density function $a(1+\gl\sin(x\qqqq))\exp(-x\qqqq)$
 with the normalizing constant $a=1/24$. Then, using $\sin(y)=(e^{\ii
  y}-e^{-\ii y})/2\ii$ and \eqref{Agamma1},
\begin{equation*}
  \begin{split}
\E X_\gl^n	
&=a\intoo x^n\bigpar{1+\gl\sin(x\qqqq)}\expe{-x\qqqq}\dd x 
\\&
=4a\intoo y^{4n+3}\bigpar{1+\gl\sin(y)}\expe{-y}\dd y 
=\frac16{\Gamma(4n+4)},
  \end{split}
\end{equation*}
for any integer $n\ge0$ and any $\gl\in[-1,1]$; thus the variables
$X_\gl$ have the same integer moments. For $\gl=0$, the same
calculation applies to non-integer $n$ as well, and shows that 
$\E X_0^s=\frac16\Gamma(4s+4)$, $-1<s<\infty$, so $X_0$ has moments \ogt.
However, this formula cannot hold for any other $\gl$ (and $s$ in an
interval), by the uniqueness \refC{C1}.

Note that $X_0\eqd Z^4$, where $Z$ has the Gamma distribution
$\Gamma(4)$, \cf{} \refE{Egamma}. A similar example is provided by
$N^6$ (or $|N|^{\ga}$ for any real number $\ga>4$) with $N\simin \N(0,1)$, 
see \cite{Berg}; indeed $|N|^\ga\eqd 2^{\ga/2}Z^{\ga/2}$ with
$Z\simin\gG(1/2)$, 
see Examples \refand{Echi2}{Echi},
and $cZ^{\gb}$ with $c>0$ and $Z\simin\gG(\gam)$
is not determined by its (integer) moments for any $\gam>0$ and $\gb>2$, 
see \eg{} \cite[Section 4.10]{Gut},
so it too provides a counter example.
See
also
\cite{Slud}.
\end{remark}

\begin{remark}
  Many of the examples in \refS{Sex} are  infinitely divisible, for
  example the Gamma distribution $\gG(\ga)$, 
$W_\ga$   \cite[p.~26]{Bondis}, 
$\tP_\ga$ and thus $P_\ga$ \cite[p.~26]{Bondis}, 
$L_\ga$ \cite{Pillai}.
We do not know whether there are any interesting
  connections between moments \ogt{} and infinite divisibility.
\end{remark}

\begin{remark}
  It is possible to consider, more generally,  moments of the form
  \eqref{gamma} where $a_j,b_j, \cx_k,\dx_k$ may be complex (and
  appearing in conjugate pairs to make the function real for real $s$). 
We have not pursued this extension and do not know whether there are
  any interesting results or examples for this class.
A trivial example is the following.

Let $X$ have a two-point distribution with $\P(X=x_1)=\P(X=x_2)=1/2$,
 where $0<x_1<x_2<\infty$.
Then $Y\=\log X$ too has a two-point distribution with 
$\P(Y=y_1)=\P(Y=y_2)=1/2$ where $y_j=\log x_j$, $j=1,2$.
Let $d\=\E Y=(y_1+y_2)/2$ and $\gb\=(y_2-y_1)/2\pi$; thus $y_1,y_2=d\pm\pi\gb$.
Then, using \eqref{Asin},
\begin{equation*}
  \E X^s
=\E e^{sY}
=e^{sd}\cosh(\pi\gb s)
=e^{sd}\sin\Bigpar{\frac\pi2+\pi\gb s\ii}
=\frac{\pi e^{sd}}
{\Gamma\bigpar{\frac12+\ii\gb s}\Gamma\bigpar{\frac12-\ii\gb s}} .
\end{equation*}
\end{remark}

\appendix

\section{Some standard formulas}\label{Appa}

For the readers' (and our own)
convenience we here collect some well-known formulas
for the Gamma function, see \eg{} \cite[Chapter 6]{AS}.
Recall that $\Gamma$ is a meromorphic function in the complex plane,
with simple poles at the non-negative integers $0, -1, -2, \dots$ and
no zeros, so $1/\Gamma$ is an entire function.
\begin{align}
  \Gamma(s)&=\intoo t^{s-1} e^{-t} \dd t, \qquad \Re s>0;
\label{Agamma}
\\
\Gamma(z+1)&=z\Gamma(z);
\label{As+1}
\\
\Gamma(2z)&=\pi\qqw 2^{2z-1}\Gamma(z)\Gamma(z+\tfrac12);
\label{Adouble}
\\
\Gamma(3z)&=(2\pi)\qw 3^{3z-1/2}\Gamma(z)\Gamma(z+\tfrac13)\Gamma(z+\tfrac23);
\label{Atriple}
\\
\Gamma(mz)&= (2\pi)^{-(m-1)/2}
m^{mz-1/2}\prod_{j=0}^{m-1}\Gamma(z+\tfrac jm)
\label{Am}
\\
\Gamma(z)\Gamma(1-z)&=\frac{\pi}{\sin(\pi z)};
\label{Asin}
%\\\Gamma(1+z)\Gamma(1-z)&=\frac{\pi z}{\sin(\pi z)};
\\
\intoo t^{s-1} e^{-at} \dd t &= a^{-s} \Gamma(s),
\qquad \Re s>0,\,\Re a>0;
\label{Agamma1}
\\
\intoi t^{s-1}(1-t)^{u-1}\dd t &= \frac{\gG(s)\gG(u)}{\gG(s+u)}, 
\qquad \Re s,\,\Re u>0;
\label{Abeta}
\\
\intoo t^{s-1}(1+t)^{-v}\dd t &= \frac{\gG(s)\gG(v-s)}{\gG(v)}, 
\qquad \Re v>\Re s>0;
\label{Abeta2}
\end{align}

Equation \eqref{Agamma1} yields by
Fubini--Tonelli  a relation between the Laplace transform and
negative moments for any positive random variable $X$:
\begin{equation}\label{Alap}
  \intoo t^{s-1} \E e^{-tX}\dd t = 
  \E\intoo t^{s-1} e^{-tX}\dd t = 
\Gamma(s)\E X^{-s},
\qquad s>0.
\end{equation}

The residue 
$\res{-n}(\gG)=(-1)^n/n!$
(an easy consequence of \eqref{As+1}). 
Thus, more generally, for any complex
$a\neq0$ and $b$,
\begin{equation}
  \label{Ares}
\res{z=-(n+b)/a}\bigpar{\gG(az+b)}
=a\qw \frac{(-1)^n}{n!}.
\end{equation}

Stirling's formula says that for all complex $z$ in a sector $\abs{\arg
  z}<\pi-\eps$ avoiding the negative real axis 
  \begin{equation}
	\label{Astir}
\log\gG(z) = (z-\half)\log z - z +\lpi+\Oqwa z,
  \end{equation}
where the logarithm $\log z$ is the principal value with imaginary
part in $(-\pi,\pi)$. (Here, $\eps>0$ is arbitrary, but the implicit
  constant in the $O$ term depends on $\eps$.)
By differentiating \eqref{Astir} twice we find, for $|\arg z|<\pi-\eps$
(\eg, for $\Re z>0$),
\begin{align}
  \frac{\dd}{\dd z}\bigpar{\log\gG(z)}&=\log z+O\bigpar{|z|\qw},\label{Astir'}
\\
  \frac{\dd^2}{\dd z^2}\bigpar{\log\gG(z)}&=\frac1z+O\bigpar{|z|\qww}.
\label{Astir''}
\end{align}
(Note that also the error term may be differentiated since the
functions are analytic in a larger sector and we may use Cauchy's
estimate for the derivative.)

\newcommand\jour{\emph}
\newcommand\book{\emph}
\newcommand\vol{\textbf}
\newcommand\SPA{\jour{Stochastic Process. Appl.} } 

\newcommand\webcite[1]{\hfil
   \penalty0\texttt{\def~{{\tiny$\sim$}}#1}\hfill\hfill}
\newcommand\arxiv[1]{\webcite{arXiv:#1.}}

\end{document}